\documentclass[10pt,letterpaper]{article}
\usepackage[top=0.85in,left=2.75in,footskip=0.75in]{geometry}

\usepackage{amsmath,amssymb}

\usepackage{changepage}

\usepackage[utf8x]{inputenc}

\usepackage{textcomp,marvosym}

\usepackage{cite}

\usepackage{nameref,hyperref}

\usepackage[right]{lineno}

\usepackage{microtype}
\DisableLigatures[f]{encoding = *, family = * }

\usepackage[table]{xcolor}

\usepackage{array}

\newcolumntype{+}{!{\vrule width 2pt}}

\newlength\savedwidth

\newcommand\thickhline{\noalign{\global\savedwidth\arrayrulewidth\global\arrayrulewidth 2pt}%
\hline
\noalign{\global\arrayrulewidth\savedwidth}}

\usepackage{setspace}
\raggedright
\usepackage[aboveskip=1pt,labelfont=bf,labelsep=period,justification=raggedright,singlelinecheck=off]{caption}

\makeatletter
\renewcommand{\@biblabel}[1]{\quad#1.}
\makeatother

\usepackage{lastpage,fancyhdr,graphicx}
\usepackage{epstopdf}
\fancyheadoffset[L]{2.25in}
\fancyfootoffset[L]{2.25in}
\usepackage{algorithmic}
\usepackage{algorithm}
\newtheorem{theorem}{Theorem}
\newtheorem{lemma}{Lemma}
\newenvironment{proof}{\vspace{1ex}\noindent{\bf Proof}\hspace{0.5em}}
	{\hfill\qed\vspace{1ex}}
\def\qed{\hfill \vrule height 6pt width 6pt depth 0pt}
\newcommand{\Prob}{\mathbb P}
\newcommand{\E}{\mathbb E}
\newcommand{\R}{{\mathbb R}}
\newcommand{\onetom}{1,\cdots,m}

\newcommand{\TT}{\mathbb T}
\newcommand{\NN}{\mathbb N}
\begin{document}
\vspace*{0.2in}

\begin{flushleft}
{\Large
\textbf\newline{On a framework of data assimilation for neuronal networks} % Please use "sentence case" for title and headings (capitalize only the first word in a title (or heading), the first word in a subtitle (or subheading), and any proper nouns).
}
\newline
\\
Wenyong Zhang\textsuperscript{1},
Boyu Chen\textsuperscript{1},
Jianfeng Feng\textsuperscript{2,3},
Wenlian Lu\textsuperscript{1,4,5,6*}
\\
\bigskip
\textbf{1} School of Mathematical Sciences, Fudan University, No. 220 Handan Road, Shanghai, China
\\
\textbf{2} Institute of Science and Technology for Brain-Inspired Intelligence, Fudan University, Shanghai, China
\\
\textbf{3} Key Laboratory of Computational Neuroscience and Brain-Inspired Intelligence, Fudan University, Ministry of Education, Shanghai, China
\\
\textbf{4} Shanghai Center for Mathematical Sciences, Fudan University, Shanghai, China
\\
\textbf{5} Shanghai Key Laboratory for Contemporary Applied Mathematics, Fudan University, Shanghai, China
\\
\textbf{6} Key Laboratory of Mathematics for Nonlinear Science, Fudan University, Ministry of Education, Shanghai, China
\\

\bigskip

* wenlian@fudan.edu.cn

\end{flushleft}
\section*{Abstract}
When handling real-world data modeled by a complex network dynamical system, the number of the parameters is always even much more than the size of the data. Therefore, in many cases, it is impossible to estimate these parameters and however, the exact value of each parameter is frequently less interesting than the distribution of the parameters may contain important information towards understanding the system and data. In this paper, we propose this question arising by employing a data assimilation approach to estimate the distribution of the parameters in the leakage-integrate-fire (LIF) neuronal network model from the experimental data, for example, the blood-oxygen-level-dependent (BOLD) signal. Herein, we assume that the parameters of the neurons and synapses are inhomogeneous but independently identical distributed following certain distribution with unknown hyperparameters. Thus, we estimate these hyperparameters of the distributions of the parameters, instead of estimating the parameters themselves. We formulate this problem under the framework of data assimilation and hierarchical Bayesian method, and present an efficient method named Hierarchical Data Assimilation (HDA) to conduct the statistical inference on the neuronal network model with the BOLD signal data simulated by the hemodynamic model. We consider the LIF neuronal networks with four synapses and show that the proposed algorithm can estimate the BOLD signals and the hyperparameters with good preciseness. In addition, we discuss the influence on the performance of the algorithm configuration and the LIF network model setup.

% Please keep the Author Summary between 150 and 200 words
% Use first person. PLOS ONE authors please skip this step.
% Author Summary not valid for PLOS ONE submissions.
\section*{Author summary}
With the great development of brain neuroimaging technology such as MRI, in recent decades, we can understand the brain function more generally than before and the advanced computer technology makes it possible to construct computational neuronal network models according to the brain neuroimaging data. However, the number of the parameters of the neuronal network may be very huge, usually larger than the sample size of observation data. Therefore, it is difficult to estimate all parameters accurately with avoiding the over-fitting problem at the same time. This paper proposes a novel framework, named Hierarchical Data Assimilation(HDA), based on an assumption that the observation data of the network model is independent of the specific parameters but depends on the distribution of these parameters, proves the convergence (asymptotic) and well-posedness of the HDA estimation approach in the view of hierarchical Bayesian inference, and illustrate the efficiency and power of the algorithms to estimate the BOLD signal and the hyperparameters by numerical simulating experiments. We think that this proposed algorithm is a promising method for a statistical inference method to integrate the large-scale full-brain computational model with the experimental brain data.

%\linenumbers
\section*{Introduction}\label{Introduction}
Over recent decades, research within neuroscience focuses on simulating the human brain to discover the neural mechanisms underlying perceptual, cognitive and motor functions, and the dynamical properties of the human brain that appears to guide the development of Artificial Intelligence, Psychiatry, Information Science, Cognition Sciences, etc. Besides, experiments in animals brain earlier used perturbative approaches such as implant electrodes or surgical operations, but these approaches are difficult and unethical in humans. Thus it is necessary to construct a computational neuronal network model with computers or silicon chips and design experiments on the virtual brain~\cite{2018Perturbation}. For modeling of neural dynamics, there are also quantities of the model such as MP neuron~\cite{1943A}, Hodgkin–Huxley (HH) model~\cite{1952A}, neural mass models~\cite{1972Excitatory} and other models referred to in~\cite{2018The}. Intuitively, we should choose the most appropriate model but it is impossible to use implanted electrodes to record the rates and timing of action potentials in the human brain. Fortunately, with the development of human neuroimaging research, alterations in brain activity can be detected non-invasively and at the whole-brain level using electroencephalogram (EEG), magnetoencephalogram(MEG) or functional MRI (fMRI)~\cite{2018Moving, 0Electrical, 2014Modern, 1998Colloquium, 2017Razi}. These techniques make simulating the brain computationally come true: a whole-brain neuronal network has been modeled by a normal Hopf bifurcation with resting-state MEG brain activity and BOLD(blood oxygen level dependent) signals~\cite{2017Single, 2015The}. All of these experiments used the structural connectivity (SC) to fit the functional connectivity (FC) from neuroimaging data and it is shown that the model parameters can be adjusted to describe different brain states~\cite{2017Increased}. Therefore, the model estimation based on specified datasets is an essential issue in neuronal network simulation.

Data assimilation(DA) is an efficient method aiming at simulating a dynamic system by a data-tuned mathematical model which was widely used in extensive fields. The earliest successful application of DA was to handle data in atmospheric sciences and weather forecasting which was comprehensively described in the book~\cite{1981Dynamic}. Then, the methodology of data assimilation gradually extends its application areas especially in geochemistry~\cite{1999A}, oceanography~\cite{2008A}, financial econometrics~\cite{2011Particle}, etc. With the development of DA, several specific algorithms have been established. The classic algorithm named Kalman filter (KF) provided by Kalman~\cite{1960A} is to deal with the Gaussian linear dynamic system. Some generalizations of the Kalman filter methodology such as extended Kalman filter (EKF)~\cite{2007Optimal}, unscented Kalman filter (UKF)~\cite{1997A}, ensemble Kalman filter (EnKF)~\cite{2006Data} were lately established and proved to be effective to nonlinear systems. And from the viewpoint of monte carlo (MC) sampling, there are also many significant algorithms including bootstrap filter (BF)~\cite{2002Novel}, unscented particle filter (UPF)~\cite{2001The} and sampling importance resampling (SIS)~\cite{1992Bayesian} etc. Recently, Kody Law develops a unified mathematical perspective of DA in which a Bayesian formulation of the problem provides the bedrock for the derivation and development of these algorithms~\cite{2015Data}.

Although DA has played a fundamental role in these areas, it has not been systemically applied in neuroscience. Path integral approach~\cite{2010State,2011Abarbanel,2012Abarbanel} and self-organizing state space model~\cite{2012A} have been used in parameter estimation based on HH model~\cite{2013Predicting}. Article~\cite{2016Comparing} compared estimation performance of EnKF, BF, and the optimal sequential importance re-sampling (OPT-SIRS) methods in HH model. However, since the large parameter space of HH model makes robust analysis difficult, Leaky Integral-Fired model (LIF) is another common neuron model~\cite{Stein1965A, 2006A}. Some parameter estimation algorithms for LIF model have been put forward~\cite{2008Parameter}, however, LIF model is seldom simulated via the framework of DA.

On the other hand, it is worth considering the inhomogeneity of neurons, that is, the values of the same parameter for different neurons should be different. For example, the decay/rising time scales are different from neurons; the synaptic conductances are different from synapses. It was proved by experiments that the firing activities of neurons have variability with respect to neurons \cite{1996CHAOS}. So, it is natural to believe that the parameters of neurons should be set inhomogeneous in the neuronal network model to better mimic the dynamics. In recent years, with the increase of the neuronal network scale and the computing power, the data or parameters of the model are more and more huge. Thus hierarchical Bayesian model as a novel estimation framework has been proposed and applied in many fields including computer vision~\cite{2005A}, seasonal hurricane simulation~\cite{2004A}, and solving boundary value problems in atmospheric stream function fields~\cite{2003Hierarchical}. In neuroscience, the Bayesian inference is commonly used to construct a hierarchical graphical structure in modeling cortical processing~\cite{Friston2010The,Lee2003Hierarchical}, and the size of parameters of each layer is usually small or large but with special hierarchical structures such as tree-structured models~\cite{2017Hierarchical}.

In this article, we establish a framework named hierarchical data assimilation (HDA) for estimating the hyperparameter of the dynamic equation model based on hierarchical Bayesian inference, give two algorithms as derivatives of DA methods to estimate the hyperparameter, and show the convergence of our algorithms from the view of maximum likelihood estimation (MLE). Because the number of hyperparameters and the observation have the same order of magnitude much less than model parameters, it avoids over-fitting effectively. In addition, we describe the neuron network model~\cite{2001Effects} and hemodynamic model~\cite{2000Nonlinear, 2021Dynamical} in the following and put our approaches to a test. Finally, we show it is efficient to estimate the hyperparameter of the model and fit the real fMRI signals.

\section*{Methods}\label{Methods}

In this section, we give a detailed description of our HDA frameworks, neuronal networks, and other methods to show how we can apply HDA to neuronal networks.

Here we give a brief introduction of our assimilation task and the experimental results are displayed in section \nameref{Results}. Firstly, the connections of neurons in networks are generated according to the specific network topology. Using Euler iterative method, we transform the neurons ODE model and hemodynamics model into the discrete-time equations. Then hierarchical neuron networks assimilation estimates the hyperparameters by iterating two processes: simulation, computing the hidden states including model parameters and states by numerical calculation, and assimilation, adjusting the model hidden states by HDA methods such as ensemble Kalman filter (EnKF) according to Bold observation. It should be noted that EnKF methods used ensemble samples to calculate the variance of hidden states and in this paper, if not otherwise specified, ensemble samples are 100 neuron networks with the same topology structure.

\subsection*{Hierarchical Data Assimilation}\label{sec:HDA}
Hierarchical Data Assimilation (HDA) is the crucial approach proposed in this paper based on Data Assimilation(overview in \nameref{S1_Appendix}). When the number of the parameters in the network model is much more than the size of the observation data, it is impossible to estimate these parameters and the exact values of these parameters are frequently of no interest but the distribution of the parameters may contain important information about the system and data. Here, we establish HDA framework by introducing hyper-parameter $h$ to describe the distribution of parameters, which may contain the important information about these parameters.

Let us consider the following general coupled dynamic systems:
\begin{eqnarray}
\frac{dx_{i}}{dt}=f_{i}(x_{i},x^{i},\theta_{i}),~i=\onetom.
\label{Eq:coupled dynamic}
\end{eqnarray}
Here, $x_{i}\in\R^{n}$ stands for the state vector of agent $i$, such as the collection of states in LIF model and Balloon-Windkessel model, as mentioned in \nameref{Methods}, $x^{i}=[x_{j}:~j\ne i]$ stands for the vectors from the other nodes, $\theta_{i}\in\R^{p}$ stands for the parameters of the evolution law of node $i$, which is related to states of other nodes. The observation of $x=[x_{1},\cdots,x_{m}]$ is:
\begin{eqnarray}
y_{t_{l}}=H\circ x_{t_{l}}
\label{Eq:coupled observation}
\end{eqnarray}
where $H$ is an observation operator that highly reduces the dimensions, i.e., $q\ll m\times n$, and $\{t_{l}\}_{l=1}^{L}$ are the sampling time points. When the parameters $\theta_{i}$ are assumed to be independently identical distributed (I.I.D) with the density function $\psi(\theta|h)$:
\begin{align*}
\theta_{i}\sim \psi(\cdot|h)
\end{align*}
where $h\in\R^{r}$ is the hyperparameter vector. Giving a long recording of $y$ from a large population of nodes, HDA aims at estimating the hyperparameter $h$. Define
\begin{align*}
\theta=[\theta_{1},\cdots,\theta_{n}]^{\top}
\end{align*}
to be regarded as the i.i.d. samples from $\phi(\cdot|h)$.

When dealing with real dataset, the observation is discrete-time with time step $\Delta t$, which is assumed to be sufficiently small (high sampling frequency), and so we consider the difference version of the above equations $(\ref{Eq:coupled dynamic}, \ref{Eq:coupled observation})$, by the basic Euler approach, as follows:
\begin{eqnarray}
x(t+1)=F(x(t),\theta,h),\quad
y(t)=H(x(t)),
\label{Eq:Euler network}
\end{eqnarray}
with $x(t) = [x_1(t), \cdots, x_m(t)]$. The question is how to estimate the hyperparameters from the history of observation sequence $Y_t=\{y_t\}_{i\in\TT}$, that is, how to estimate the posterior densities $\Prob(h|Y_t)$. An intuitive idea under this problem formulation is we can regard the distribution of the recording dynamic time series as a map from the distribution, in terms of the parameter of this distribution named hyperparameter, instead of from the exact values of all parameters of the model.

Noticed that:
\begin{equation*}
p(h|Y_{t})=\frac{p(h)}{p(Y_t)}p(Y_t|h)=\frac{p(h)}{p(Y_t)}\int p(Y_t|\theta, h)p(\theta|h)d\theta.
\end{equation*}

We assume that the evolution of the states does not depend the specific parameters but depend the hyperparameter. That means that there exists a infinitesimal $\varepsilon(h, m)$ such as $\frac{1}{m}$, $\forall \theta_1, \theta_2$, satisfying
\begin{align}
\int|p(Y_t|\theta_1, h)-p(Y_t|\theta_2, h)|p(\theta_1|h)d\theta_1<\varepsilon(h, m),
\label{Eq:assumption1}
\end{align}
and $\varepsilon(h, m)$ converges to zero almost everywhere when $m$ goes to infinity.
This assumption is not trivial but depends on the property of the observation function(see the \nameref{Remark}).

Then we have the following approximation using the I.I.D. samples $\theta$
\begin{align}
p(h|Y_{t})\approx\frac{p(h)}{p(Y_t)}p(Y_t|\theta, h), a.e.
\label{Eq:assumption2}
\end{align}

According to the evolution of the states by the transition probability density and Markov property, we have
\begin{align*}
p(Y_t|\theta, h)=&\int p(Y_t|X_t)p(X_t|\theta, h) dX_t\\
=&\int p(Y_t|X_t)\prod_{t=1}^{T}p(x_t|x_{t-1}, \theta, h)p(x_0)dx_0\cdots dx_t\\
=&\int p(Y_t|X_t)\prod_{t=1}^{T}p(x_t|x_{t-1}, \theta)p(\theta|h)p(x_0) dx_0\cdots dx_t
\end{align*}
From the perspective of DA, the hyperparameters $h$ are considered as states following stochastic differential equation $\frac{d}{dt}h(t)=\sigma dB_t$, which evolve discretely via random walks:
\begin{align*}
h(t+1)=h(t)+\zeta(t)
\end{align*}
where $\zeta_{t}$ is a white Gaussian noises with components independent and small variances $\sigma_{\zeta}^{2}$ which are known preliminarily. The evolution of the parameters $\theta$ is tricky. Here, we also consider the following law:
\begin{align*}
\theta(t+1)=\Phi(\theta(t), h(t+1), h(t), \eta(t)),
\end{align*}
where $\eta(t)$ are independent (w.r.t time and component) and distributed so that $\theta(t+1)$ follows the distribution with parameters $h(t+1)$. The specific realization methods and more details are discussed in next subsection.

Comparing with the approaches in DA, we consider the following two methods: ensemble Kalman filter and particle filter. Let write $[x(t), h(t)]$ by $X(t)$ for simplicity. Here we show the EnKF algorithm combined with our HDA as Algorithm \ref{algo:enkf} and other details are referred to in \nameref{S2_Appendix}.

\begin{algorithm}
	\renewcommand{\algorithmicrequire}{\textbf{Input:}}
	\renewcommand{\algorithmicensure}{\textbf{Output:}}
	\caption{EnKF combined with HDA}
	\label{algo:enkf}
	\begin{algorithmic}[1]
		\REQUIRE $N, [\mu, C_0], F_{0:T}(\cdot), \Sigma_h, \Sigma_x, \eta_{0:T}, [H, \Gamma], y_{0:T}$.
		\ENSURE $\{X_{0:T}^{(n)}\}_{n=1}^N$.
        \STATE \textbf{draw} $X_0^{(n)}\sim \mathcal{N}(\mu_0, C_0)), \theta_0^{(n)}\sim\Prob(\cdot|h_0^{(n)}), \hat{h}_0^{(n)}=h_0^{(n)}, \qquad\forall n=1:N$
		\FOR{$t=1:T$}
        \STATE \textbf{draw} $\hat{h}^{(n)}_t\sim \mathcal{N}(h_{t-1}^{(n)}, \Sigma_h),\qquad\forall n=1:N$.
        \STATE $\theta^{(n)}_t=\Phi(\theta^{(n)}_{t-1}, \hat{h}_{t}^{(n)}, \hat{h}_{t-1}^{(n)}, \eta(t)),\qquad\forall n=1:N$.
		\STATE \textbf{draw} $\hat{x}^{(n)}_t\sim \mathcal{N}(F_{t-1}(x_{t-1}^{(n)}, \theta_{t}^{(n)}), \Sigma_x),\qquad\forall n=1:N$.
		\STATE $\hat\mu_t\leftarrow\frac{1}{N}\Sigma_{n=1}^{N}\hat{X}_t^{(n)}$.
		\STATE $\hat{C}_t\leftarrow\frac{1}{N-1}\Sigma_{n=1}^{N}(\hat{X}_t^{(n)}-\hat\mu_t)(\hat{X}_t^{(n)}-\hat\mu_t)^T$.
        \STATE $S_t\leftarrow H\hat C_tH^T+\Gamma$.
        \STATE $K_t\leftarrow \hat C_tH^TS_t^{-1}$.
        \STATE \textbf{draw} $y_{t}^{(n)}\sim\mathcal{N}(y_t, \Gamma),\qquad\forall n=1:N$.
		\STATE $\delta_t^{(n)}\leftarrow y_t^{(n)}-H\hat x_t^{(n)},\qquad\forall n=1:N$.
		\STATE $X_t^{(n)}\leftarrow\hat X_t^{(n)}+K_j\delta_t^{(n)},\qquad\forall n=1:N$.
		\ENDFOR
		\STATE \textbf{return} $\{X_{0:T}^{(n)}\}_{n=1}^N$.
	\end{algorithmic}
\end{algorithm}

\subsubsection*{Remark}
In our HDA framework, a hierarchical structure is used to replace estimating the distribution of parameters by estimating hyperparameters. Thus we have not changed the topology structure of the network. Since we assume that each group of parameters sampled from the same distribution should have the same system observation, the parameters initialization of each ensemble sample is independent in Algorithm \ref{algo:enkf}. And we improve the accuracy of hyperparameters estimation by increasing the number of parameters rather than the times of parameters resampled. So it allows us not only to expand the network scale but also to simulate quickly by using limited computing resources.

% % % % % % % % % % % % % % % % % % % % % % %
\subsection*{Computational Neuronal network model}

The computational neuron model is generally a nonlinear operator from a set of input synaptic spike trains to an output axon spike train, described by three components: The subthreshold equation of membrane potential that describes the transformation from the synaptic currents of diverse synapses; the synaptic current equation describes the transformation from the input spike trains to the corresponding synaptic currents; the threshold scheme gives setup of the condition for triggering a spike by membrane potential.

We, following \cite{2001Effects}, consider the leakage integrate-and-fire (LIF) model as neuron. A capacitance-voltage (CV) equation describes the membrane potential of neuron $i$, $V_i$, when it is less than a given voltage threshold $V_{th,i}$:
\begin{align}
C_i \dot{V}_i=-g_{L,i} (V_i-V_L )+\sum_u I_{u,i} +I_{ext,i}, \qquad V_i<V_{th,i}
\label{eq:LIF1}
\end{align}
where $C_i$ is the capacitance of the neuron membrane, $g_{L,i}$ is the leakage conductance, $V_L$ is leakage voltage, $I_{u,i}(t)$ stands for the synaptic current at the synapse type $u$ of neuron $i$ and $I_{ext,i}$ is the external stimulus from the environment and we set $I_{ext,i}=0~nA$ in this paper. Here, we consider at least four synapse types: $AMPA$, $NMDA$, $GABA_A$ and $GABA_B$~\cite{2016A}.

As far the synapse model, We consider an exponentially temporal convolution as this map:
\begin{align}
\begin{array}{ll}
I_{u,i}&=g_{u,i}(V_{u,i}-V_i)J_{u,i}\\
\dot{J}_{u,i}&=-\frac{J_{u,i}}{\tau_{u,i}}+\sum_{k,j}w_{ij}^u\delta(t-t_k^j)
\end{array}
\label{eq:LIF2}
\end{align}
Here, $g_{u,i}$ is the conductance of synapse type $u$ of neuron $i$, $V_u$ is the voltage of synapse type $u$, $\tau_{u,i}$ is the time-scale constants of synapse type $u$ of neuron $i$, $w_{ij}^u$ is the connection weights from neuron $j$ to $I$ of synapse type $u$, $\delta(\cdot)$ is the Dirac-delta function and $t_k^j$ is the time point of the $k-th$ spike of neuron $j$.

When $V_i=V_{th,i}$ at $t=t_n^i$, neuron registers a spike at time point $t_n^i$ and the membrane potential is reset at $V_{rest}$ during a refractory period:
\begin{align}
V_i(t)=V_{rest},t\in[t_k^i,t_k^i+T_{ref}].
\label{eq:LIF3}
\end{align}
After then, $V_i$ is governed by CV equation again.

Based on the above neuron model, the neuronal network is composed of $N_{E}=4/5 N$ excitatory neurons, denoted by $\mathcal N_{E}$ and $N_{I}=1/5 N$ inhibitory neurons, denoted by $\mathcal N_{I}$. Let $W_u=[w^u_{ij}]_{i,j=1}^{N}$ be the adjacent matrix of the synapse type $u$ where $w^u_{ij}\neq0$ means there is a link from neuron $j$ to $i$ and the conductance of this link is $w^u_{ij}$, otherwise $w^u_{ij}=0$. Each neuron has $D$ neighbors and the network graph topology is generated by the following rule: each neuron randomly chooses a synaptic junction from $4/5D$ excitatory neurons as excitatory synapses and $1/5D$ inhibitory neurons as inhibitory synapses. Thus the edge of each $u$-synapse graph is $N\times D$. For convenience, we transform these ODE into the iterated map by the Euler method of time step 1 msec and calculate the mean of neuron activity in the neuronal network as output:
\begin{align}
z(t)=\frac{1}{N}\sum_{i\in N, k\in \NN}\delta(t-t^i_k).
\label{eq:LIF4}
\end{align}

% % % % % % % % % % % % % % % % % % % % % % %
\subsection*{BOLD signal and hemodynamical model}

FMRI is the widely used towards human objects as BOLD signal that indirectly measures the neural activity in a low spatial and temporal resolution. A useful mathematical model is the hemodynamical model that takes the neural activity quantified by the spike rate of a pool of neurons and outputs bold signal. Let $z_t$ be the time series of neural activity from Eq~(\ref{eq:LIF4}) and hemodynamical model can be generally written as $\dot{g}=G(g, z)$.
Here $g$ denotes the states related to blood volume and blood oxygen consumption. And the bold signal is read as a function of $g$ i.e., $y(t)=h(g(t))$. There are diverse hemodynamical models and herein, we introduce the Balloon-Windkessel model~\cite{2000Nonlinear} as shown:
\begin{align}
\left\{
\begin{array}{lr}
\dot{s}_i=\varepsilon z_i-\kappa s_i-\gamma(f_i-1)\\
\dot{f}_i=s_i\\
\tau\dot{v}_i=f_i-v_i^{\frac{1}{\alpha}}\\
\tau\dot{q}_i=\frac{f_i(1-(1-\rho)^{\frac{1}{f_i}})}{\rho}-\frac{v_i^{\frac{1}{\alpha}}q_i}{v_i}\\
y_i=V_0[k_1(1-q_i)+k_2(1-\frac{q_i}{v_i})+k_3(1-v_i)]\\
\end{array}
\right.
\label{eq:BOLD}
\end{align}

Here, $i$ is the label of ROI/voxel, $s_i$ is vasodilatory signal, $f_i$ is blood inflow, $v_i$ is blood volume and $q_i$ is deoxyhemoglobin content. $\varepsilon, \kappa, \tau, \gamma, \alpha, \rho, V_0, k_1, k_2, k_3$ are the constant parameter of the Balloon-Windkessel model. All parameters of constant term in neuronal network model and hemodynamic model are set in Table~\ref{Table:1}.

% Place tables after the first paragraph in which they are cited.
\begin{table}[!ht]
\begin{adjustwidth}{0in}{0in} % Comment out/remove adjustwidth environment if table fits in text column.
\centering
\caption{
{\bf Parameters in Neuronal Network Model and Hemodynamic Model.}}
\begin{tabular}{|c|c+c|c|}
\hline
\multicolumn{2}{|c+}{\bf Neuronal Network Model} & \multicolumn{2}{c|}{\bf Hemodynamic Model}\\ \thickhline
$\bf symbol$ & $\bf value$ & $\bf symbol$ & $\bf value$\\ \hline
$C_i$ & 1$\mu f$ & $\varepsilon$ & 200\\ \hline
$g_{L,i}$ & 0.03$mS$ & $\kappa$ & 1.25\\ \hline
$V_L$ & -75$mV$ & $\gamma$ & 2.5\\ \hline
$V_{th}$ & -50$mV$ & $\alpha$ & 0.2\\ \hline
$V_{rest}$ & -65$mV$ & $\rho$ & 0.8\\ \hline
$[V_{AMPA,i},V_{NMDA,i},V_{GABA_a,i},V_{GABA_b,i}]$ & [0,0,-70,-100]$mV$ & $V_0$ & 0.02\\ \hline
$[\tau_{AMPA,i},\tau_{NMDA,i},\tau_{GABA_a,i},\tau_{GABA_b,i}]$ & [2,40,10,50] & $\tau$ & 1\\ \hline
$w_{i,j}^u$ & $\sim$ rand(0,1) & $k_1$ & 5.6\\ \hline
$T_{ref}$ & 5$ms$ & $k_2$ & 2\\ \hline
~ & ~ & $k_3$ & 1.4\\ \hline
\end{tabular}
\begin{flushleft}
Table shows parameters of constant term in Eq~(\ref{eq:LIF1}), Eq~(\ref{eq:LIF2}), Eq~(\ref{eq:LIF3}) and Eq~(\ref{eq:BOLD}). $C_i$ is the capacitance of the neuron membrane, $g_{L,i}$ is the leakage conductance, $V_L$ is leakage voltage, $V_{th}$ is threshold voltage, $V_{rest}$ is rest voltage after spike, $V_u$ is the voltage of synapse type $u$, $\tau_{u,i}$ is the time-scale constants of of synapse type $u$, $u\in \{AMPA, NMDA, GABA_a, GABA_b\}$, $w_{i,j}$ is I.I.D from a uniform distrition over $(0, 1)$. Parameters in hemodynamic model are referred in detail in~\cite{2000Nonlinear}.
\end{flushleft}
\label{Table:1}
\end{adjustwidth}
\end{table}

Furthermore, to mimic the time series of BOLD signals that were measured by the fMRI, we conduct a down-sampling process of the recording data:
\begin{align*}
TS(l)=y(t_l)+\eta(l),~l=1,2,\cdots,L
\end{align*}
where $\eta(l)$ is a white Gaussian noise with $cov(\eta(l)\eta(s))=\sigma_{n}\delta(l-s)$ with $\sigma_{n}=10^{-8}$, $t_{l}$, $l=1,2,\cdots,L$ are sampling time points that are periodic with frequency $1.25$~HZ, as same as frequency in fMRI.

% % % % % % % % % % % % % % % % % % % % % % %

\subsection*{Sampling parameters from the hyperparameter}

Since our algorithm is based on the framework of Hierarchical Bayesian Inference, it is necessary to introduce how hyperparameter determines parameters in the algorithm. As mentioned in section \nameref{sec:HDA}, the evolution of the parameters $\theta$ should be tricky because it should ensure that the parameter follows the distribution with hyperparameter $h$. Here, we give two methods to adjust the parameter.

The first one is simple and effective by using the cumulative density function (cdf) and its inverse function. As we can see, the update rule of parameter in the previous subsection is according to the following equation:
\begin{align}
\tilde{g}_{i}(t+1)= \mathcal{F}_{h(t+1)}^{-1}\circ\mathcal{F}_{h(t)}\left(g_{i}(t)\right),
\label{par_ite}
\end{align}
where $\mathcal{F}_{c}(\cdot)$ denotes the cdf by exponential distribution with mean $c$. Therefore if the parameter follows the specific distribution, we just use the cdf to map $\theta$ to $\vartheta = \mathcal{F}_{h(t)}(\theta)$, following uniform distribution and hold random variables $\vartheta$ in this probability space. Then the map $\vartheta\rightarrow\hat\theta=\mathcal{F}_{h(t+1)}^{-1}(\vartheta)$ ensure $\hat\theta$ following the distribution with new hyperparameter $h(t+1)$.

And the other method, finding a transform rule from the $\theta_{t-1}$ to $\theta_t$ based on the old probability density function(pdf) $p_0(\cdot)$ and the new one $q(\cdot)$ according to MCMC method, is to deal with the situation when the cdf is difficult to be calculated numerically, such as Gussian distribution or Gamma distribution. Denote the parameters of the last iteration as $\theta=(\theta_1, \cdots, \theta_N)$ and the new parameters as $\hat\theta$. We aim to make the distribution of new parameter as close as possible to new distribution $\mu_q$. By probability theory, the pdf of $\hat\theta$ becomes $p_{\hat{\theta}}(x)=\Sigma_{n=1}^{N}\frac{1}{N}\mathbb{I}_{\{x=\hat{\theta}\}}$. The algorithm is as follow Algorithm \ref{algo:mcmc}.

\begin{algorithm}
	\renewcommand{\algorithmicrequire}{\textbf{Input:}}
	\renewcommand{\algorithmicensure}{\textbf{Output:}}
	\caption{Transformation based on MCMC method}
	\label{algo:mcmc}
	\begin{algorithmic}[1]
		\REQUIRE $\theta_0=(\theta_0^1, \cdots, \theta_0^N), p_0(\cdot), q(\cdot), T, \Sigma_0$
		\ENSURE $\hat\theta$
		\FOR{$t=0:T-1$}
        \FOR{$n=1:N$}
		\STATE \textbf{draw} $\theta^n_{t+1}\sim \mathcal{N}(\theta^n_{t}, \Sigma_0),$
		\STATE \textbf{calculate} $q_{(\theta^n_{t})}$ and $q_{(\theta^n_{t+1})}$
		\STATE \textbf{draw} $\mu \sim Union(0,1)$
		\STATE \textbf{let} $\theta^n_{t}=\theta^n_{t+1},\qquad$ if $\mu \leqslant \frac{q_{(\theta^n_{t+1})}}{q_{(\theta^n_{t})}}$
        \ENDFOR
        \STATE \textbf{denote} $\theta_{t+1}=(\theta^1_{t+1}, \cdots, \theta^N_{t+1})$
		\ENDFOR
		\STATE \textbf{return} $\hat\theta=\theta_{T}$
	\end{algorithmic}
\end{algorithm}

Fig~\ref{Fig:MCMC}(a) illustrates performance of Algorithm~\ref{algo:mcmc} at $N=1000$ with $T=1000$ to generate Gamma distribution $\Gamma(2, 1)$. According to this method, it allows us to choose different distribution of parameters to model the brain network (See more details in \nameref{Discussion}).

\begin{figure}[ht]
\includegraphics[width=.9\textwidth]{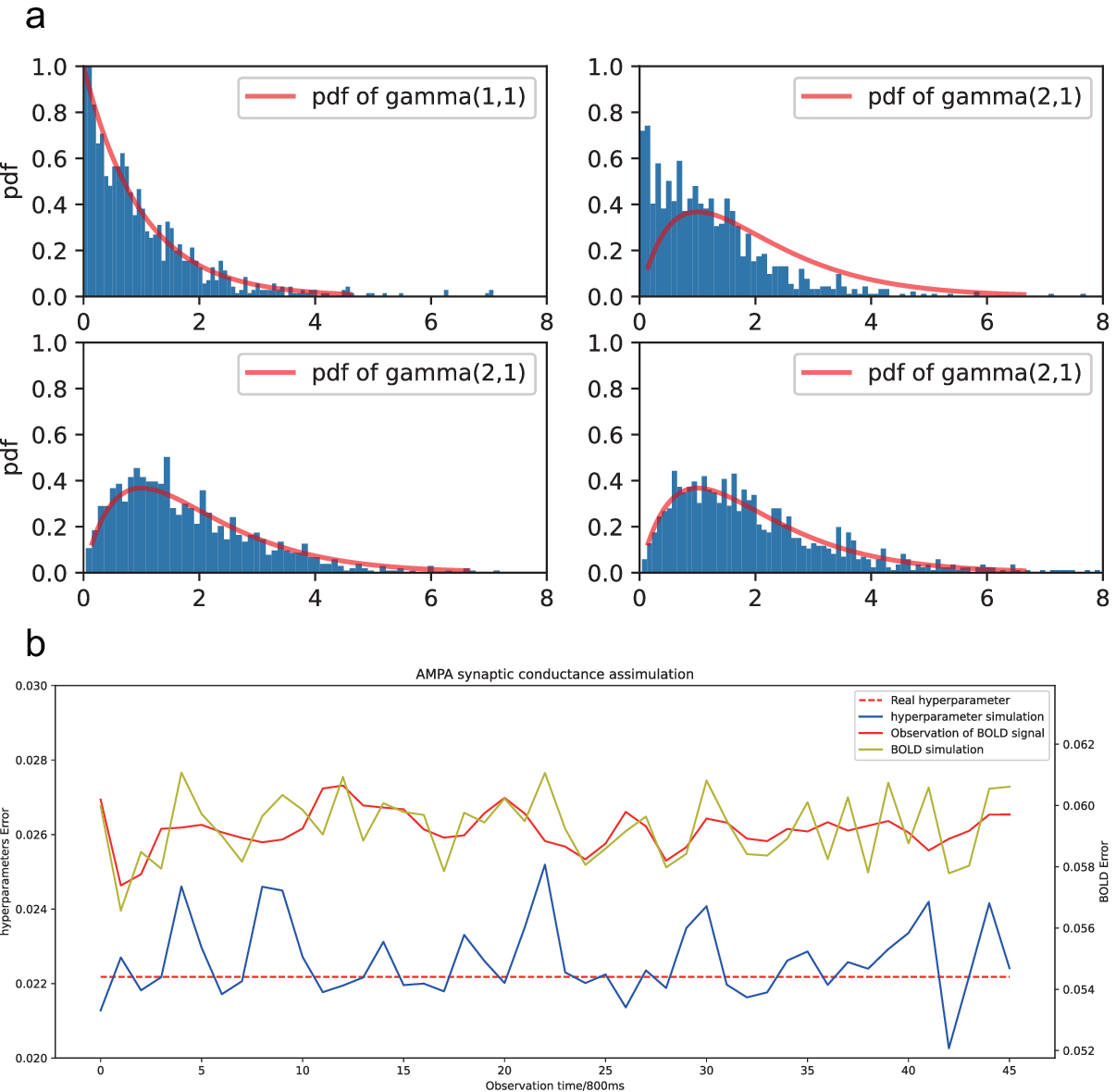}\\
\caption{\bf Sampling parameters from the hyperparameter via MCMC Method.}
{\bf a}: These four pictures show the change from exponential distribution to Gamma distribution during MCMC method iteration and the iteration are 0, 1, 10, 100 respectively.   {\bf b} Based on MCMC method, HDA can also assimilite arbitrarily distribution of parameters effectively.
\label{Fig:MCMC}
\end{figure}

% % % % % % % % % % % % % % % % % % % % % % %
% % % % % % % % % % % % % % % % % % % % % % %
\section*{Results}\label{Results}
%In this section, we utilise our HDA algorithm to estimate the parameter of a large-scale synaptically coupled neuronal network model composed of $N$ neurons. As we show in previous section, the parameter scale of neuronal network may be larger than the observation scale. And the blood oxygen level dependent (BOLD) signal, as observation of neuronal network, can be measure by functional magnetic resonance imaging (fMRI), which are time series on each voxel/region and of very low frequency. Thus, we do not consider the hemodynamics at the level of milliseconds but periodic sampling time with a fixed frequency(1.25hz in this paper). We have given a detailed description of the model in Method.
%it has been shown that HDA is efficial to simulation the bold sigmal and synaptic transmitter conductivity in the above sections. But

% % % % % % % % % % % % % % % % % % % % % % %
\subsection*{Asymptotic and well-posedness analyses}

In this subsection, we give mathematical proof of the convergence (asymptotics) of the HDA estimation process in the view of MLE.

Let $Y$ be an $\mathbb{R}^m$ -valued random variable, whose probability distribution is defined by a series of random variable $\Theta_m=[\theta_1, \cdots, \theta_m]$. Meanwhile, $[\theta_1, \cdots, \theta_m]$ called parameters are i.i.d. following a specific distribution $\psi(\cdot|h)$ with hyperparameter $h$. Thus the pdf of $Y$ can be represented by conditional probability:
\begin{equation}
\label{pdf}
\Prob (y=Y|\Theta_m,h).
\end{equation}
It is assumed that $(Y_1, \cdots, Y_n)$ called available data is a set of independent realizations of $Y$, for example coming from observations or numerical simulations. To simplify the presentation, it is assumed that there is no noise in the data. We try to prove hyperparameter $h$ can be estimated as the size of parameters and the times of observation of data goes to infinity. Since we only have information that parameters are sampled from a given distribution, it can not get post distribution accurately. so it is necessary to construct a likelihood function that depends on an unknown hyperparameter $h$ firstly.

Let $\hat{h}_{n,m}$ be the solution of the following optimization problem,
\begin{align}
\hat{h}_{n,m} = \arg\min L(h,Y_1,\cdots,Y_n,\Theta_m),
\label{mle1}
\end{align}
in which $Y_i$ is sample of state $Y$ and $L(\cdot)$ is the log-likelihood function defined by
\begin{align}
L(h,Y_1,\cdots,Y_n,\Theta_m) &=-\frac{1}{N}\sum_{i=1}^N L(h,Y_i,\Theta_m^i)\notag\\
&= -\frac{1}{N}\sum_{i=1}^N \log \Prob(Y_i|h,\Theta_m^i)
\label{mle2}
\end{align}
If the log-likelihood function is differentiable and $\hat{h}_{n,m}$ is a global maximum of the optimization problem, we have
\begin{equation}
\frac{\partial}{\partial h}L(h,Y_1,\cdots,Y_n,\Theta_m)|_{h = \hat{h}_{n,m}} = 0.
\label{mle3}
\end{equation}

For the equation $\ref{mle3}$, we have a Taylor expansion for $h$ in the neighborhood $\mathcal{N}(\tilde{h})$ about the point $h = \tilde{h}$, we assume $h\in\R$ for simplify, as follow
\begin{equation}
\begin{aligned}
\frac{\partial}{\partial h}L(h,Y_1,\cdots,Y_n,\Theta_m)&=\frac{\partial}{\partial h}L(h,Y_1,\cdots,Y_n,\Theta_m)|_{h = \tilde{h}}\\
&+(h-\tilde{h})\frac{\partial^2}{\partial h^2}L(h,Y_1,\cdots,Y_n,\Theta_m)|_{h = \tilde{h}}\\
&+\frac{1}{2}\zeta(h-\tilde{h})^2H(Y_1,\cdots,Y_n, \tilde{h}),
\label{Taylor expansion}
\end{aligned}
\end{equation}
where $\zeta\in(0,1)$ and for all $Y$,
\begin{align*}
|\frac{\partial^3}{\partial h^3}L(h,Y_1,\cdots,Y_n,\Theta_m)|\leqslant H(Y_1,\cdots,Y_n, h).
\end{align*}

We assume:
\begin{itemize}
\item[R1.] For each $h\in\mathcal{N}(\tilde{h})$,
\begin{align*}
\frac{\partial}{\partial h}L(h,Y,\Theta_m),\frac{\partial^2}{\partial h^2}L(h,Y,\Theta_m),\frac{\partial^3}{\partial h^3}L(h,Y,\Theta_m),
\end{align*}
exist, for all $Y$.
\item[R2.] For each $h\in\mathcal{N}(\tilde{h})$, there exist functions $f(Y), g(Y), H(Y)$ (possibly depending on h) such that the relations
\begin{align*}
|\frac{\partial}{\partial h}\Prob(Y|h,\Theta_m)|\leqslant f(Y), |\frac{\partial^2}{\partial h^2}\Prob(Y|h,\Theta_m)|\leqslant g(Y), |\frac{\partial^3}{\partial h^3}L(h,Y,\Theta_m)|\leqslant H(Y)
\end{align*}
a.s. hold, for all $Y$, and
\begin{align*}
\int f(Y)dY<\infty, \int g(Y)dY<\infty, \E_Y[H(Y)]<\infty.
\end{align*}
\item[R3.] For all $\Theta_m$ and $\Phi_m$ following the distribution $\psi(\cdot|h)$ with same hyperparameter $h$, then there exist $k, K, K_{\infty}$ satisfying
\begin{align*}
%0<k\leqslant\int[\frac{\partial}{\partial h}L(h, Y, \Theta_m)]^2\Prob(Y|h, \Theta_m)dY\leqslant K<\infty, \\
&0<k\leqslant\int[\frac{\partial}{\partial h}L(h, Y, \Theta_m)]^2\Prob(Y|h, \Phi_m)dY\leqslant K<\infty, &a.s.\\
&0<k\leqslant\int[\frac{\partial}{\partial h}L(h, Y, \Theta_m)]^4\Prob(Y|h, \Phi_m)dY\leqslant K<\infty, &a.s.\\
&\lim_{m\rightarrow\infty}\int[\frac{\partial}{\partial h}L(h, Y, \Theta_m)]^2\Prob(Y|h, \Theta_m)dY= K_{\infty}, &a.s.
\end{align*}
\item[R4.] The Hellinger distance between two distributions is defined by
\begin{align*}
d_{Hell}(\Prob_1, \Prob_2) = \big(\frac{1}{2}\int(\sqrt{p_1(y)}-\sqrt{p_2(y)})^2dy\big)^{\frac{1}{2}}.
\end{align*}
We assume there exists a distribution $\Prob(Y|h)$ satisfies that for arbitrary distribution $\Prob(Y|h, \Theta_m)$,
\begin{align*}
d_{Hell}(\Prob(Y|h, \Theta_m), \Prob(Y|h)) \leqslant \varepsilon(m, h)\xrightarrow[m\rightarrow\infty]{a.s.} 0,
\end{align*}
i.e. $\varepsilon(m, h)$ converges to zero almost everywhere when m goes to infinity.
Thus, for arbitrary two distributions $\Prob(Y|h, \Theta_m)$ and $\Prob(Y|h, \Phi_m)$, we have
\begin{align*}
d_{Hell}(\Prob(Y|h, \Theta_m), \Prob(Y|h, \Phi_m)) \leqslant 2\varepsilon(m)\xrightarrow[m\rightarrow\infty]{a.s.} 0.
\end{align*}
\end{itemize}

When $\{Y_i\}$ drawn from pdf $\Prob(\cdot|\tilde{\theta}, \tilde{h})$ with hyperparameter $\tilde{h}$ and parameter $\tilde{\theta} = [\tilde{\theta}_1, \cdots, \tilde{\theta}_m]$ which are the real hyper-parameter and parameters of the specific dynamic system, Theorem ~\ref{Thm:1} shows $\tilde{h}$ and $\hat{h}_{n,m}$ are very close to each other as $m,n\rightarrow\infty$ and the prove is in \nameref{S3_Appendix}.

\begin{theorem}
Assume conditions (R1-R4), the solution of equation ($\ref{mle3}$) $\hat{h}_{n,m}$ and the real hyperparameter $\tilde{h}$ satisfy that $|\tilde{h} - \hat{h}_{n,m}|$ converges to zero as $m,n$ goes to infinity.
\label{Thm:1}
\end{theorem}

On the other hand, if we have two different simulation of the data, $Y^{'}$ and $Y^{''}$ respectively, the difference of hyperparameter estimator based on different observations, $|\tilde h^{'}-\tilde h^{''}|$, can be controlled by the $|Y^{'}-Y^{''}|$~\cite{2015Data}. It means our estimator given by MLE depends continuously on the observation data, shown in Theorem~\ref{Thm:2} and the prove is in \nameref{S4_Appendix}.

\begin{theorem}
Let $x$ denote all model states including hyperparameter and Bold model states, $\mu$ and $\mu'$ be the posterior probability distributions of model states associated with two different data sets $Y=y_{1:\TT}$ and $Y^{'}=y'_{1:\TT}$ respectively, $\mu_0$ be the prior measure on $x$. Assume that $\mathbb{E}_{\mu_0}w(x_{1:\TT}) < \infty$, where $w(x_{1:\TT})$ is given by
\begin{align}
w(x_{1:\TT}) = \Sigma_{t=1}^\TT(1+|H(x_t)|^2).
\end{align}
Then there exists $c = c(r)$ such that for all $|y_{1:\TT}|, |y'_{1:\TT}| < r$,
\begin{align}
d_{Hell}(\mu, \mu')\leqslant c|y_{1:\TT}-y'_{1:\TT}|.
\end{align}
\label{Thm:2}
\end{theorem}

% % % % % % % % % % % % % % % % % % % % % % %
\subsubsection*{Remark}\label{Remark}
In this paper, we use HDA framework to simulate the hyperparameters of ordinary differential equations with assumption~\ref{Eq:assumption1}, and here we will show the assumption can be achieved by a specific observation function of the system. Assumption~\ref{Eq:assumption1} essentially represents not each In this paper, we use HDA framework to simulate the hyperparameters of ordinary differential equations with assumption~\ref{Eq:assumption1}, and here we will show the assumption can be achieved by a specific observation function of the system. Assumption~\ref{Eq:assumption1} essentially represents not each parameters but the hyperparameters are vital to the observation series $y_t$. And a necessary condition is that as the number of parameters $m$ goes to infinite, each parameter $\theta$ in ODE has little effect to observation, such as $\lim_{m\rightarrow\infty} \frac{\partial y_t}{\partial \theta} = 0$. It is obvious that choosing the mean value function of all states as an observation function satisfies the condition and here we give a simple model.

Let us consider one of the simplest case of discrete time. $x_{i}(t)$ be an i.i.d. Gaussian process: $x_{i}(t)=\mu_{i}+\varepsilon_{i,t}$, $i=\onetom$, and $t=1,\cdots,T$, where $\mu_{i}\in\R$, $\sigma_{i}>0$ and $\varepsilon_{i,t}$ is an standard white Gaussian process which are independent with respect to $i$ and $t$. Suppose that all $\mu_{i}$ are independently identically Gaussian distributed with $N(\mu_0,1)$. The observation $y(t)=\frac{1}{m}\sum_{i=1}^{m}x_{i}(t)$ is the average of all $x_{i}$. After a simple calculation, we have $y(t)=\frac{1}{m}\sum_{i=1}^{m}(\mu_{i}+\varepsilon_{i,t})\sim N(\mu_0, \frac{2}{m})$, $\lim_{m\rightarrow\infty} \frac{\partial y(t)}{\partial \mu_i} = 0$, $\E y(t)=\mu_0$. And if $\mu$ and $\mu'$ are both i.i.d. with $N(\mu_0,1)$, $\lim_{m\rightarrow\infty}|P(y'|\mu',\mu_0)-P(y|\mu,\mu_0)|=0$. Thus the observation satisfies the condition\ref{Eq:assumption1}. From this simple model, we can also find that if the task is to estimate hyperparameter $\mu_0$, larger $m$ makes $y$ and $\mu_0$ more closer. In other words, the estimation statistics of the hyperparameter converges to the real value (with respect to the invariant measure) as $m$ (number of equations/states) goes to infinity. The hyperparameters are vital to the observation series $y_t$. And a necessary condition is that as the number of parameters $m$ goes to infinite, each parameter $\theta$ in ODE has little effect to observation, such as $\lim_{m\rightarrow\infty} \frac{\partial y_t}{\partial \theta} = 0$. It is obvious that choosing the mean value function of all states as an observation function satisfies the condition and here we give a simple model.

Let us consider one of the simplest case of discrete time. $x_{i}(t)$ be an i.i.d. Gaussian process: $x_{i}(t)=\mu_{i}+\varepsilon_{i,t}$, $i=\onetom$, and $t=1,\cdots,T$, where $\mu_{i}\in\R$, $\sigma_{i}>0$ and $\varepsilon_{i,t}$ is an standard white Gaussian process which are independent with respect to $i$ and $t$. Suppose that all $\mu_{i}$ are independently identically Gaussian distributed with $N(\mu_0,1)$. The observation $y(t)=\frac{1}{m}\sum_{i=1}^{m}x_{i}(t)$ is the average of all $x_{i}$. After a simple calculation, we have $y(t)=\frac{1}{m}\sum_{i=1}^{m}(\mu_{i}+\varepsilon_{i,t})\sim N(\mu_0, \frac{2}{m})$, $\lim_{m\rightarrow\infty} \frac{\partial y(t)}{\partial \mu_i} = 0$, $\E y(t)=\mu_0$. And if $\mu$ and $\mu'$ are both i.i.d. with $N(\mu_0,1)$, $\lim_{m\rightarrow\infty}|P(y'|\mu',\mu_0)-P(y|\mu,\mu_0)|=0$. Thus the observation satisfies the condition\ref{Eq:assumption1}. From this simple model, we can also find that if the task is to estimate hyperparameter $\mu_0$, larger $m$ makes $y$ and $\mu_0$ more closer. In other words, the estimation statistics of the hyperparameter converges to the real value (with respect to the invariant measure) as $m$ (number of equations/states) goes to infinity.

% % % % % % % % % % % % % % % % % % % % % % %
\subsection*{Simulation of LIF neuronal network}

We use our HDA Algrothms to simulate LIF neuron network mentioned in \nameref{Methods} of $N=1000$ with degree $D=20$ up to $T=40000~msec$ by the Euler method of time step $1~msec$. Thus, we have a BOLD-signal time series of $50$ time points. The brain region activity is recorded as the sum of neuron activity for all $i\in N$:
\begin{align*}
z(t)=\frac{1}{N}\sum_{i\in N, k\in \NN}\delta(t-t_i^k)+\eta(t).
\end{align*}
where $\eta(t)$ is a white Gaussian noise with $cov(\eta(t)\eta(s))=\sigma_{n}\delta(t-s)$ with $\sigma_{n}=10^{-8}$. After getting brain region activity, we use Balloon-Windkessel model mentioned above to generate the BOLD signal. Then we focus on how to estimate the hyperparameters $h$ of synaptic conductance $g_{ui}$ from the observation series $\{y(t_l)\}_{l=1}^T$. To estimate these hyperparameters, parameters of synaptic conductance is sampled from the exponential distribution with the given hyperparameter independently and the procedure(see in \nameref{S2_Appendix}) is similar to Algorithm~\ref{algo:enkf}. When we choose one of the hyperparameters to assimilate, such as $g_{AMPA, i}$ $i\in\mathcal N$, all other parameters including $g_{u, i}(u\neq AMPA)$ are assumed to be known. All other neuron network model parameters are recapitulated in Table \ref{Table:1} and the details of algorithm configuration are in \nameref{Discussion}.

Fig~\ref{Fig:result}(a, b) illustrates BOLD error and hyperparameter error when we choose different combinations of hyperparameters to assimilate. The main criterion is how close the BOLD signal generated by toy-model with specific hyperparameters $h$ is to the BOLD signal simulation with hyperparameters simulation under specify configuration of algorithm parameters, as follow:
\begin{align*}
\varepsilon_{BOLD}&=\frac{1}{T}\sum_{t=1}^{T}\frac{|BOLD_{simulation}(t)-BOLD_{real}(t)|^2}{|BOLD_{real}(t)|^2}.\\
\varepsilon_{h}&=\frac{1}{T}\sum_{t=1}^{T}\frac{|h_{simulation}(t)-h_{real}(t)|^2}{|h_{real}(t)|^2}.
\end{align*}
As shown, the best performance in single hyperparameter assimilation is $u=AMPA$, and in the conbination of excitatory and inhibitory synaptic conductance hyperparameters assimilation is $u=NMDA, GABA_b$.

% Place figure captions after the first paragraph in which they are cited.
\begin{figure}[ht]
\includegraphics[width=.8\textwidth]{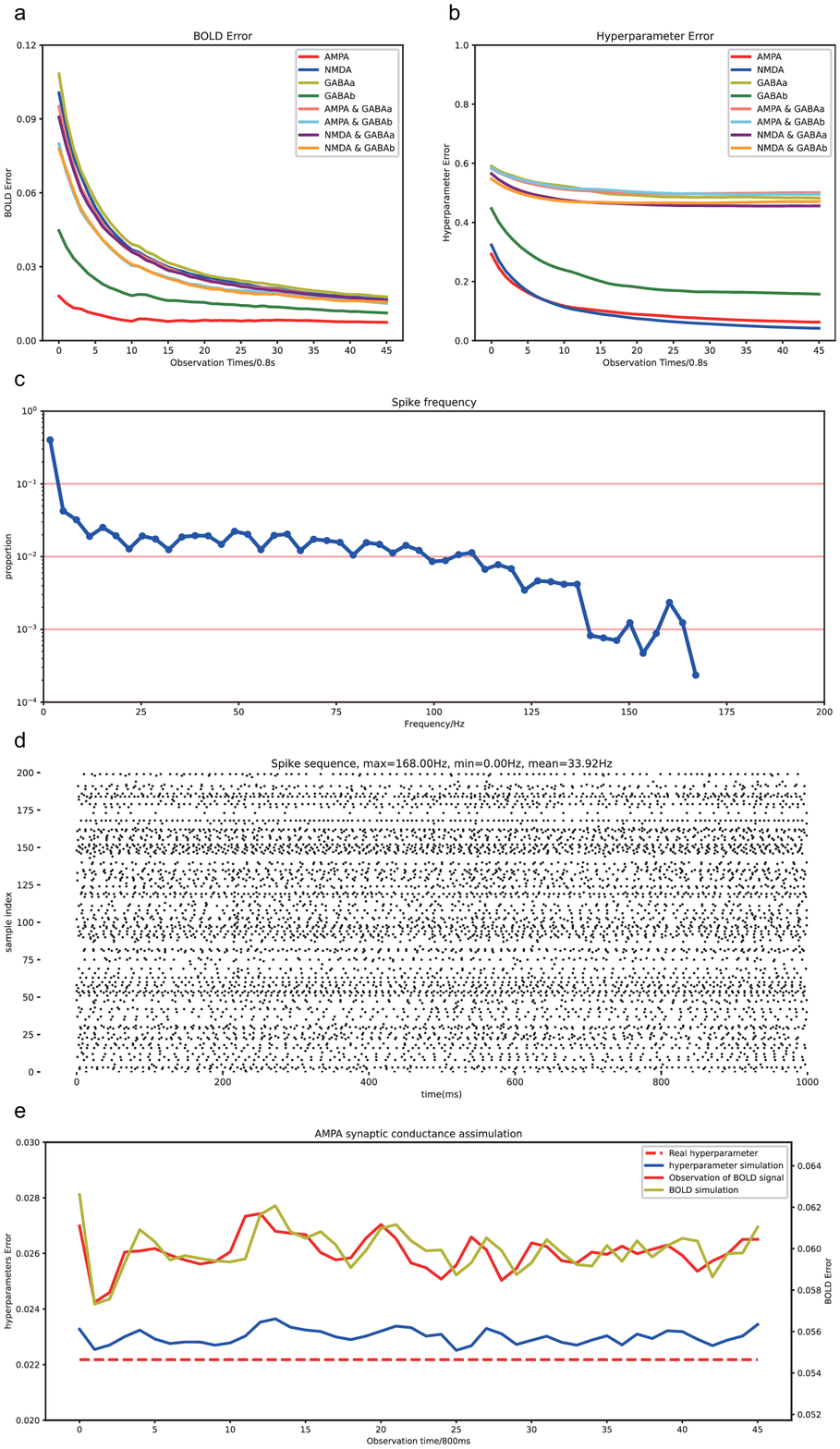}\\
\caption{\bf BOLD error and hypermeter error under different combinations of hyperparameter.}
{\bf a, b}: The difference of BOLD singles and hyperparameters between simulation and real are shown respectively on single synaptic conductance hyperparameter assimilation or the conbination of excitatory and inhibitory synaptic conductance hyperparameters assimilation. The error funcition is defined in above. {\bf c}: Sliding window is used to calculate the fire rate of sampling neurons and 40 percent of neurons have a firing rate of less than 5 Hz. {\bf d}: Although the observation is at the level of second, the spike of sample neurons can be recorded at each millisecond based on our model. {\bf e}: Bold signal and hyperparmeter serie of real data are both close to simulation under AMPA synaptic conductance assimulation.
\label{Fig:result}
\end{figure}

% % % % % % % % % % % % % % % % % % % % % % %
\subsection*{Performance of HDA}

Here we show the basic behavior of the network during a particular assimilation trial, $AMPA$ synaptic conductance assimilation. We record the spikes of all neurons at each millisecond, use a sliding window to average the spike every 800 msec, and calculate the spike frequency distribution shown in Fig~\ref{Fig:result}(c). It is easy to see that spike sequence and frequency distribution are similar to the record from the human brain, which shows the availability of our neuron network and method from the side. Besides, we choose 200 neurons to record their activities over 1000 milliseconds and we can see there is no synchronization or excessive activation in sample neurons in Fig~\ref{Fig:result}(d). Besides, in Fig~\ref{Fig:result}(a, b), it is shown that $AMPA$ synaptic conductance assimilation has the best performance which means the simulation signal is most close to the real Bold signal. We show Bold signal and hyperparameter series of HDA simulation and real data under $AMPA$ synaptic conductance assimilation in Fig~\ref{Fig:result}(e).

% % % % % % % % % % % % % % % % % % % % % % %
%\subsection*{Robustness}
%In HDA framework, we have an assumption that the observation of the network is independent of the specify parameters, not depends on distribution of parameters. So, in this subsection, we will show the assumption is realizable. We choose $AMPA$ synaptic conductance hyperparameter to assimilate and record the estimation hyperparameter sequence. Then we resample the parameters many times according to the hyperparameter sequences with the same neuron network topology and re-run the simulation. As shown in Fig~\ref{Fig:result}(f), the BOLD error under different sampling are all close to the BOLD error in HDA, which means the hyperparameters given by our methods has good robustness in our neuron network model.

\subsection*{Using MCMC to sample parameters from the hpyerparameter}

The Algorithm~\ref{algo:mcmc} (in \nameref{Methods}) makes it be realized to assimilate arbitrary parameters distribution with known probability density function. Without losing generality, we choose Gamma distribution instead of exponential distribution as a sample whose cumulative distribution function is hard to calculate. It is assumed that the parameters of AMPA synaptic conductance is i.i.d Gamma distribution with $\alpha=2$ and the original distribution in HDA is Gamma distribution with $\alpha=1$. As shown in Fig~\ref{Fig:MCMC}b, with more iteration times, HDA method still assimilates the real hyperparameter and thus is still valid.
% % % % % % % % % % % % % % % % % % % % % % %
% % % % % % % % % % % % % % % % % % % % % % %
\section*{Discussion}\label{Discussion}
% % % % % % % % % % % % % % % % % % % % % % %
\subsection*{Compare the setup of the algorithms}

In this section, we explore that how different learning parameters setup (in Table~\ref{Table:discussion}) in our method and different structures of network in out model influence the simulation result by experimental data. The main criterion is how close the BOLD signal generated by toy-model is to the BOLD signal simulation under specify configuration of different parameters and how close the estimation of hyperparameter is to the real hyperparameter.

\begin{table}[ht]
\begin{adjustwidth}{-2in}{0in} % Comment out/remove adjustwidth environment if table fits in text column.
\centering
\caption{
{\bf Compare the different configuration of algorithm.}}
\begin{tabular}{|c|c|c|c|c|c|}
\hline

\multicolumn{2}{|c|}{Different configuration} & {Standard deviation of } & {Standard deviation of } & {Number of} & {Parameter of}\\
\multicolumn{2}{|c|}{of HDA Algorithm \ref{algo:enkf}} & {hyperparameter noise} & {observation noise} & {the samples} & {WS network}\\
\thickhline
~ & configuration & 0.01 & $10^{-4}$ & 5 & 0 \\ \cline{2-6} Config1 & hp error $(\times10^{-3})$ & $4.66\pm2.85$ & $683\pm265$ & $112\pm195$ & $55.2\pm65.3$ \\ \cline{2-6} ~ & bold error $(\times10^{-3})$ & $6.25\pm2.87$ & $184\pm102$ & $29.8\pm31.5$ & $3.18\pm0.02$ \\ \thickhline
~ & configuration & 0.1 & $5\times10^{-4}$ & 10 & 0.1 \\ \cline{2-6} Config2 & hp error $(\times10^{-3})$ & $4.89\pm3.27$  & $5.94\pm1.18$ & $26.7\pm58.4$ & $63.7\pm71.5$ \\ \cline{2-6} ~ & bold error $(\times10^{-3})$ & $5.70\pm2.58$  &$5.46\pm0.69$ &$9.92\pm11.5$ &$3.23\pm0.02$ \\ \thickhline
~ & configuration & 0.2 & $\mathbf{10^{-3}}$ & 50 & 0.2 \\ \cline{2-6} Config3 & hp error $(\times10^{-3})$ & $4.90\pm1.97$ & $\mathbf{5.29\pm0.79}$ & $5.71\pm1.95$ & $67.4\pm65.2$ \\ \cline{2-6} ~ & bold error $(\times10^{-3})$ & $4.87\pm1.07$ & $\mathbf{4.39\pm0.21}$ & $4.59\pm0.65$ & $3.21\pm0.02$ \\ \thickhline
~ & configuration & $\mathbf{0.5}$ & $5\times10^{-3}$ & $\mathbf{100}$ & 0.3 \\ \cline{2-6} Config4 & hp error $(\times10^{-3})$ & $\mathbf{5.29\pm0.79}$ & $4.18\pm0.61$ &$\mathbf{ 5.29\pm0.79}$ & $67.9\pm55.5$ \\ \cline{2-6} ~ & bold error $(\times10^{-3})$ & $\mathbf{4.39\pm0.21}$ & $4.10\pm0.26$ & $\mathbf{4.39\pm0.21}$ & $3.20\pm0.03$ \\ \thickhline
~ & configuration & 1 & $10^{-2}$ & 200 & 0.5 \\ \cline{2-6} Config5 & hp error $(\times10^{-3})$ & $5.85\pm0.85$ & $2.97\pm0.64$ & $5.19\pm0.87$ & $74.5\pm58.5$ \\ \cline{2-6} ~ & bold error $(\times10^{-3})$ & $4.78\pm0.35$ & $5.04\pm0.61$ & $4.32\pm0.11$ & $3.18\pm0.02$ \\ \thickhline
~ & configuration & 2 & $5\times10^{-2}$ & 300 & 0.8 \\ \cline{2-6} Config6 & hp error $(\times10^{-3})$ & $8.24\pm1.51$ & $7.72\pm1.87$ & $5.47\pm0.45$ & $87.3\pm98.7$ \\ \cline{2-6} ~ & bold error $(\times10^{-3})$ & $6.65\pm0.72$ & $23.4\pm4.18$ & $4.25\pm0.07$ & $3.03\pm0.01$ \\ \thickhline
~ & configuration & 5 & $10^{-1}$ & 500 & 1 \\ \cline{2-6} Config7 & hp error $(\times10^{-3})$ & $109\pm142$ & $13.2\pm2.34$ & $5.36\pm0.37$ & $69.7\pm102$ \\ \cline{2-6} ~ & bold error $(\times10^{-3})$ & $70.6\pm87.8$ & $54.9\pm8.75$ & $4.28\pm0.09$ & $3.12\pm0.02$ \\ \thickhline
\end{tabular}
\begin{flushleft}
Table shows different configurations of algorithm parameters discussed in this section to exhibit the error simply and intuitively. The highlight of the table is the final configurations used in \nameref{Results}.
\end{flushleft}
\label{Table:discussion}
\end{adjustwidth}
\end{table}

\subsubsection*{Standard deviation of noise}
When we use Enkf to estimate the state of the real system or ground truth model, we regard the hyperparameters as system states as mentioned before, and therefore it is necessary to add Gaussian noise in the model samples and observation dataset. It is obvious that bigger covariance of noise causes more randomness in the learning process, thus we should choose proper covariance to make sure the learning efficiency. Table~\ref{Table:discussion} shows the hyperparameter error and bold error under different covariance of model noise $\sigma_h$ and different covariance of observation noise $\sigma_o$. In the previous section, we set $\sigma_h=1, \sigma_o=10^{-6}$ in Fig~\ref{Fig:result}(e) and this configuration is best according to the Table~\ref{Table:2}, considering hyperparameters and BOLD error both.

\subsubsection*{Number of the samples}
In the traditional Enkf method, it is necessary to use ensemble samples to calculate the covariance of states. Thus in our algorithm, ensemble samples are also used and it is obvious that more samples cause more accuracy during simulating bold signals and estimating the hyperparameters. Table~\ref{Table:discussion} shows the hyperparameters and BOLD error under the different sizes of samples. In our trial, we use a NVIDIA Telsa V100 GPU to simulate 20 samples of 1000 neurons with network degree 20, and 5 GPU to realize 100 samples as reasonable ensemble number defaults.\\

\subsection*{Effect of network topologies}
In the previous part of this paper, we assumed the neuron network is a random network with a specific degree. For one thing, the random network can improve randomness to make assumption~\ref{Eq:assumption1} easily realized, for another, the random network can reduce the difference between different samples in EnKF. Here we try other network structures such as WS-small-world network~\cite{1998Collective} and BA-scale-free network~\cite{1999Albert}. In WS-small-world network, the parameter of the network means the probability of rewiring each edge, thus larger parameter means the network has a more random connection. In BA-scale-free network, the number of edges to attach from a new node to existing nodes is set to 10. As shown in Table~\ref{Table:discussion} and Fig~\ref{Fig:BA and pf}(a), BOLD can be simulated with small error but hyperparameters simulation is worse than random network.

\begin{figure}[ht]
\includegraphics[width=.8\textwidth]{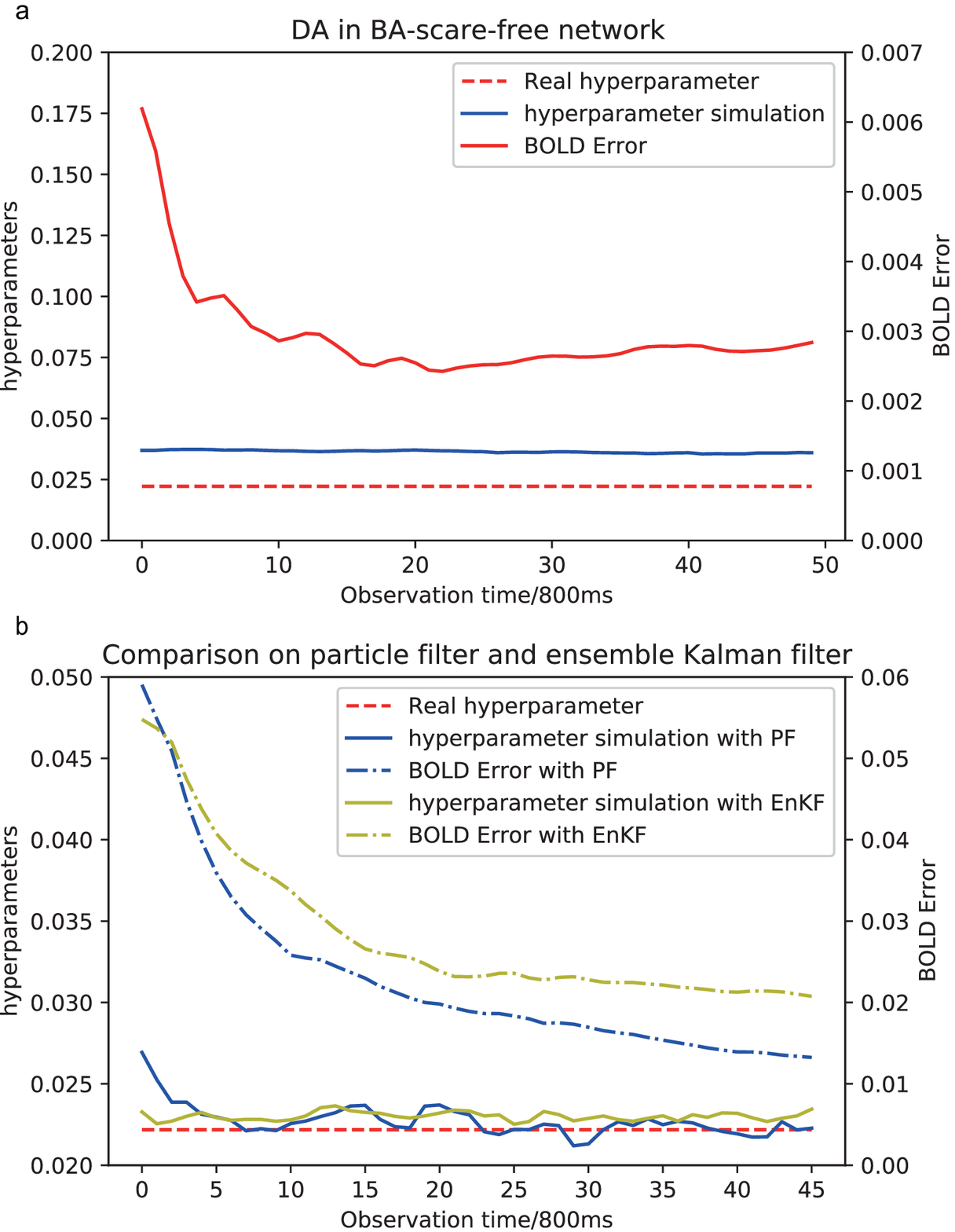}\\
\caption{\bf BOLD error and hypermeter error under different combinations of hyperparameter.}
{\bf a}: When generating BA-scale-free network, the new nodes are regard as postsynaptic neurons so that most neurons have same number of parent nodes. In this case, the observation is greatly affected by hub nodes thus hyperparameter is hard to estimate accurately. {\bf b}: We replace ensemble Kalman filter with particle filter to assimalate the AMPA synaptic conductance. The other configurations of Algorithm remain same and the numbers of samples in two trials are both 100.
\label{Fig:BA and pf}
\end{figure}

\subsection*{HDA by particle filter}
As we showed in \nameref{Methods}, except Kalman filter, HDA framework can also combine with other filters such as particle filter. Different from Kalman filter trying to minimize mean-squared posterior error, particle filter seeks to maximize a posterior probability and the specified procedure is introduced in detail in \nameref{S2_Appendix}. Because of the complexity of the neuron model and the scale of network, particle filter usually needs more samples to calculate posterior probability and thus needs more storage space and computation. Here we apply particle filter with 100 samples to make a comparison to EnKF filter. As shown in Fig~\ref{Fig:BA and pf}(b), particle filter also has a good effect in both hyperparameters estimation and BOLD signal simulation. But when the dimension of the parameter is high, the number of the samples increases exponentially to ensure the simulation accuracy.

\subsection*{Transformation function of hyperparameters}
As shown in \nameref{S1_Appendix}, HDA framework adds hyperparameters into ODE states by following the equation, $\dot{h}=0$. However, if we use $h'=f(h)$ as new hyperparameters, $h'$ should also satisfy $\dot{h'}=\frac{\partial f}{\partial h}\dot{h}=0$. Thus, theoretically, we can choose proper hyperparameters representation by adding constraints to get better simulation results. In this paper, we replace $h$ with $h'$ to have a random walk, where
\begin{eqnarray}
h = h_{min}+\frac{(h_{max}-h_{min})\exp(\lambda h')}{1+\exp(\lambda h')},
\label{Transformation f of h}
\end{eqnarray}
$\lambda=0.1$, $h_{min}$ and $h_{max}$ are one quarter and double of ground truth of synaptic conductance respectively. In this situation, we make sure $h$ belong to $(h_{min}, h_{max})$, when random walk of $h'$ is on $\R$.
% % % % % % % % % % % % % % % % % % % % % % %
% % % % % % % % % % % % % % % % % % % % % % %
\section*{Conclusion}\label{Conclusion}
When focusing on simulation or estimation of states in large-scale network systems, the number of model parameters is frequently much larger than the number of the observed samples. Especially when investigating the human brain, it is much more difficult to measure each single neuron activity but able to collect the BOLD signals of voxels or brain regions. In this paper, we give an example (in \nameref{Discussion}) that the observation data does not depend on the specific parameters but depend on the distribution of the parameters. Thus we claim that, in many cases, the distribution of the parameters in the model rather than the exact values of these parameters are physiologically interesting in a very large-scale complex network dynamical system. By describing the distribution of random variables by introducing a hyperparameterized distribution, we proposed an HDA framework by estimating hyperparameter based on hierarchical Bayesian inference. According to the trails in previous sections, the HDA framework can estimate the hyperparameters with the true values well with small errors, and the error of the sampled BOLD signals with the true simulated data is small, too. This result is shown to be stable and robust for different parameter resamples from the estimated hyper-parameter.

In detail, we design a tricky evolution of parameters to guarantee that the parameters are independent and identically distributed throughout. When the cumulative distribution function (cdf) of the parameters is easy to be numerically calculated, as we have shown in \nameref{Methods}, the evolution is achieved by the cdf and its inverse function. However, when the cdf is complicated to be numerical calculated, Algorithm~\ref{algo:mcmc} is an effective method based on the MCMC method. We arise two HDA algorithms according to the particle filter and the ensemble Kalman filter (EnKF) approaches in data assimilation. These methods are both proved efficient for a large-scale neuron network model with the simulated BOLD signals. Furthermore, we compare the BOLD signals simulation brought out by different configurations of our algorithm and different parameters in our neuron model. Estimating synaptic conductance of $AMPA$ has the best performance in single hyperparameter assimilation, and estimating both $NMDA$ and $GABA_b$ are the best in the combination of excitatory and inhibitory synaptic conductance assimilation. For network topology, BOLD signals can be simulated well in all of WS-small-world network, BA-scale-free network and random network, while hyperparameters estimation is more accurate in a random network. Finally, we prove that the HDA algorithm can asymptotically estimate the hyperparameter feasibly under certain assumptions. 

We conclude that, the proposed HDA framework present an efficient method to statistically infer the large-scale full-brain neuronal network with inhomogeneous parameters, since the number of hyperparameter is relatively small that can avoid over-fitting results. Last but not least, the HDA is an extensible framework to equip with diverse biophysical neuronal network model other than LIF network, and diverse modules of brain data other than MRI. This could be one of our future research orients.

% % % % % % % % % % % % % % % % % % % % % % %
% % % % % % % % % % % % % % % % % % % % % % %
\section*{Supporting information}

% Include only the SI item label in the paragraph heading. Use the \nameref{label} command to cite SI items in the text.

% % % % % % % % % % % % % % % % % % % % % % %
\paragraph*{S1 Appendix.}\label{S1_Appendix}
{\bf Overview of Data Assimilation.}\\
Let $F(\cdot)\in C^1(\mathbb{R}^n,\mathbb{R}^n)$, $H(\cdot)\in C^1(\mathbb{R}^n, \mathbb{R}^q)$ and consider the following general ordinary differential equation:
\begin{eqnarray}
\frac{d}{dt}x(t)=&F(x(t)),\\
y(t)=&H(x(t)),
\label{ode1}
\end{eqnarray}
where $x\in \mathbb{R}^n$ stands for the state vectors of the dynamic system and the observation of $x$ is $y\in \mathbb{R}^q$. Here $H(\cdot)$ is called a observation map.

When $F$ is so complicated that we fail to get accurate formula solution, it is effective to numerical simulation by transforming the ODE into the iterated map, or discrete-time dynamical system. Due to numerical errors and modeling errors, we are also describe the physical system by the following discrete-time stochastic dynamical model of the form:
\begin{align}
x(t)&=F(x(t-1))+\xi_t,\\
y(t)&=H(x(t))+ \eta_{t}, \quad t\in \TT=\{1, \cdots, T\}.
\label{ode2}
\end{align}
Here $\xi=\{\xi_t\}_{t\in\mathbb{T}}$ is an i.i.d.sequence, independent of $x$ and $\eta=\{\eta_{t}\}_{t\in\mathbb{T}}$ is an i.i.d.sequence, independent of $(x, y, \xi)$.

If observational data $Y_T=\{y(t)\}_{t\in\mathbb{T}}$ is acquired, DA is a kind of iterative process by correcting the states of the dynamical model based on the observation and model information and can be derived using Bayesian statistics and recursive Bayesian estimation.
There are generally two classes of DA methods. The first one is filtering methods which update the posterior filtering distribution $\mathbb{P}(x_t|Y_{T})$ at each time and the second one is smoothing methods which aim at approximating the posterior smoothing distribution $\mathbb{P}(X_{T}|Y_{T})$. Generally speaking, DA methods is optimization to seek the posterior distribution.DA methods is optimization to seek the posterior distribution and the criteria for measuring the optimality can be referred in [book Guzzi]:
\begin{enumerate}
\item The Minimum Mean-Squared Error (MMSE) that is defined in terms of prediction error
\begin{align*}
\mathbb{E}[||x_t-\hat{x}_t||^2|y_{\TT}]=\int||x_t-\hat{x}_t||^2p(x_t|y_{\TT})dx_k,
\end{align*}
where the conditional mean $\hat{x}_t = \mathbb{E}[x_t|y_{\TT}]=\int x_t\ p(x_t|y_{\TT})dx_t$.

\item Maximum A Posteriori (MAP) that is aimed to find the mode of posterior probability $p(x_t|y_{\TT})$ which is equal to minimize a cost function.

\item The Maximum Likelihood (ML) neglect the prior, that is to find maximum likelihood estimator by maximize likelihood function $L=\mathbb{P}(y_{\TT}|x_{\TT})$.

\item The Minimax that is to minimize the maximum of cost function of the posterior probability.
\end{enumerate}

%To unify, we write the evolution (the Markov law) of the states (including the random walking of parameters) by the transition probability density
%\begin{align*}
%p[X(t+1)|X(t);\alpha(t)]=\delta(X(t+1)-\mathcal F(X(t),\alpha(t))).
%\end{align*}
%where $\mathcal F(\cdot)$ can be derived from the difference equations above, $\delta(\cdot)$ is the Dirac-delta function, and $\alpha(t)$  is the collection of all random processes in the equation, including $\epsilon$, $\eta$ and $\upsilon$.
In discrete-time stochastic dynamical model, if we assume that $F(\cdot)$ and $H(\cdot)$ are determined by the parameter vector $\theta$ and $\theta$ is regarded as a state variable with dynamic equation $\dot{\theta}=0$, DA has proved to be an efficient framework in parameter estimation.

% % % % % % % % % % % % % % % % % % % % % % %
\paragraph*{S2 Appendix.}\label{S2_Appendix}
{\bf EnKF and Particle filter.}\\
Ensemble Kalman filter as a kind of Kalman filter derivatives tries to minimize mean-squared error. EnKF is a approximate Guassian filter proposed by G. Evensen, that has a natural generalization to non-Gaussian problems. Compared with Kalman Filter approach, EnKF draws some samples of the dynamic system and uses the covariance of samples to substitute for the covariance of posterior distribution. The complete EnKF procedure is as follows.

\begin{enumerate}
\item Generate $N$ independent realisations of all random terms in the model, including the Gaussian noised/error in measurements and random walks, denoted by $\alpha^{n}(t)$, $n=1,\cdots,N$;

\item Generate $N$ realisations of the predictive density $p(X(t)|Y_{t-1})$: $\hat{X}^{n}(t)$, $j=1,\cdots,N$,  by the map $\hat{X}^{n}(t)=F(\tilde X^{n}(t-1),\alpha^{n}(t))$;

\item Calculate the mean $\mu(t) = \frac{1}{N}\sum^N_{n=1}\hat X^n(t)$

\item Calculate the covariance matrix $C(t) = \frac{1}{N-1}\sum^N_{n=1}(\hat X^n(t)-\mu(t))(\hat X^n(t)-\mu(t))^T$;

\item Calculate the Kalman gain matrix $K(t) = C(t)H^T(HC(t)H^T+\Gamma)^{-1}$;

\item Generate $N$ realizations of the filter density $p(X(t)|Y_{t})$: $\tilde X^{n}(t)=\hat X^{n}(t)+K(t)(Y_{t}-HX^n(t))$, $n=1,\cdots,N$.

\end{enumerate}

From the view of minimization principle~\cite{2015Data}, we note the posterior filter distribution proportional to $\exp(-I_{filter})$. While the Kalman filter is restricted to linear Gaussian problem, the mean in Kalman filter can be written as:
\begin{align*}
m_{j+1} &= \arg\min_v I_{KF}(v)\\
I_{KF}(v) &= \frac{1}{2}|y_{j+1}-Hv|_{\Gamma}^2+\frac{1}{2}|v-Fm_j|_{\hat{C}_{j+1}}^2,\\
\end{align*}
where $F\in R^{n\times n}, H\in R^{m\times n}$ are linear operator. When we consider the nonlinear case, the mean in EnKF is given by
\begin{align}
\label{I_filter}
m_{j+1} &= \arg\min_v I_{EnKF}(v)\\
\hat{\mu}_{j+1} &= F(m_j) \notag\\
I_{EnKF}(v) &= \frac{1}{2}|y_{j+1}-H(v)|_{\Gamma}^2+\frac{1}{2}|v-\hat{\mu}_{j+1}|_{\hat{C}_{j+1}}^2\notag.
\end{align}

Particle filter approach seeks maximum a posteriori. By Bayesian and total probability formulas, we have
\begin{align}
p(X(t)|Y_{t})&=p[X(t)|(y(t),Y_{t-1})]=
\frac{p(y(t)|X(t)|Y_{t-1})p(X(t)|Y_{t-1})}{p(y(t)|Y_{t-1})}\nonumber\\
&=\frac{p(y(t)|X(t))}{p(y(t)|Y_{t-1})}
p(X(t)|Y_{t-1})\label{bayes1}
\end{align}
noting that $p(y(t)|X(t)|Y_{t-1})=p(y(t)|X(t))$ because $y(t)$ is an observation of $X(t)$, with
\begin{align}
p(y(t)|Y_{t-1})=\int p(y(t)|X(t))p(X(t)|Y_{t-1})dX(t)\label{average}
\end{align}
Monto-Carlo filter is deployed here to construct the posterior density. In this approach, each density is approximated by the sampling data. For example, for a probability density $p(\cdot)$, let $z_{1},\cdots,z_{N}$ be $N$ independent samples from this probability distribution. Then, we have
\begin{align*}
p(z)\approx\frac{1}{N}\sum_{i=1}^{N}g(z-z_{i})
\end{align*}
where $g(\cdot)$ is a high peak function centered at $0$, for example, the Gaussian-type function: $g(z)\propto\exp(-z^{2}/(2\alpha_{g}^{2}))$ with a sufficiently small $\alpha_{g}$.

Note
\begin{align}
p(X(t)|Y_{t-1})=\int p(X(t)|X(t-1))p(X(t-1)|Y_{t-1})dX(t-1).\label{pred}
\end{align}
Beside the estimation problem $p(X(t)|Y_{t})$, which is also named {\em filter}, the prediction problem is also concerned, i.e.,  $p(x_{t}|Y_{t-1})$, and related, as indicated in (\ref{pred}).

Giving the size of sampling as $N$, which is very large,  the following bootstrap algorithm is used. Initiated with the realizations from the prior distribution,  the iteration from $p(X(t-1)|Y_{t-1})$  to $p(X(t)|Y_{t})$ is conducted as follows:
\begin{enumerate}
\item Generate $N$ independent realisations of all random terms in the model, including the Gaussian noised/error in measurements and random walks, and the Poisson point processes at synapses, denoted by $\hat{\alpha}^{j}(t)$, $j=1,\cdots,N$;

\item Generate $N$ realisations of the predictive density $p(X(t)|Y_{t-1})$: $\hat{X}^{j}(t)$, $j=1,\cdots,N$,  by the map $\hat{X}^{j}(t)=F(\tilde{X}^{j}(t-1),\hat{\alpha}^{j}(t))$;

\item Calculate the density weights $\omega^{j}(t)=p(y(t)|X(t)=\hat{X}^{j}(t))$;

\item Generate $N$ realizations of the filter density $p(X(t)|Y_{t})$: $\tilde{X}^{j}(t)$, $j=1,\cdots,N$, by re-sampling (with replacement) of $\hat{X}^{j}(t)$ with sampling probabilities propositional to $\omega^{j}(t)$.

\end{enumerate}

Thus, the filter density can be constructed by the sampling $\tilde{X}^{j}(t)$ and the prediction density by $\hat{X}^{j}(t)$.Thus, the estimative expectation of any function $\E(f(X(t))|Y_{t})$ can be calculated as follows:
\begin{align*}
\E(f(X(t))|Y_{t})\approx\frac{1}{N}\sum_{j=1}^{n}f(\tilde{X}^{j}(t)).
\end{align*}
and
the predictive expectation of any function $\E(f(X(t))|Y_{t-1})$ can be calculated as follows:
\begin{align*}
\E(f(X(t))|Y_{t-1})\approx\frac{1}{N}\sum_{j=1}^{n}f(\hat{X}^{j}(t)).
\end{align*}

% % % % % % % % % % % % % % % % % % % % % % %
\paragraph*{S3 Appendix.}\label{S3_Appendix}
{\bf Prove of the Theorem~\ref{Thm:1}.}\\
Before we proof the theorem, let proof the following lemma:

\begin{lemma}
Let $p_1, p_2$ be the probably distribution function and $f:\R^l\rightarrow\R$ be such that
\begin{align*}
\int |f(y)|^2 p_1(y)dy+\int |f(y)|^2 p_2(y)dy<\infty
\end{align*}
Then
\begin{align*}
|\int f(y)(p_1(y)-p_2(y))dy| \leqslant 2\big(\int |f(y)|^2 p_1(y)dy+\int |f(y)|^2 p_2(y)dy\big)^{\frac{1}{2}}\ d_{Hell}(p_1, p_2)
\end{align*}
\end{lemma}
\begin{proof}
\begin{align*}
|\int f(y)(p_1(y)-p_2(y))dy| &\leqslant \int |f(y)||p_1(y)-p_2(y)|dy\\
&=\int\sqrt{2}|f(y)|\sqrt{p_1}+\sqrt{p_2}|\cdot\frac{1}{\sqrt{2}}|\sqrt{p_1}-\sqrt{p_2}|dy\\
&\leqslant\big(2\int |f(y)|^2 |\sqrt{p_1}+\sqrt{p_2}|^2)dy\big)^{\frac{1}{2}}({\frac{1}{2}}\int|\sqrt{p_1}-\sqrt{p_2}|^2dy)^{\frac{1}{2}}\\
&\leqslant2\big(\int |f(y)|^2 p_1(y)dy+\int |f(y)|^2 p_2(y)dy\big)^{\frac{1}{2}}\ d_{Hell}(p_1, p_2)
\end{align*}
\end{proof}

Back to the prove of Theorem~\ref{Thm:1}.

\begin{proof}
Let
\begin{align*}
A_{n,m}&=\frac{\partial}{\partial h}L(h,Y_1,\cdots,Y_n,\Theta_m)|_{h = \tilde{h}}\\
&=-\frac{1}{n}\sum_{i=1}^n\frac{\partial}{\partial h}\log\Prob(Y_i|h,\Theta_m^i)|_{h = \tilde{h}}
\end{align*}
By(R1,R2),
\begin{align*}
\int\frac{\partial}{\partial h}\Prob(Y|h,\Theta_m)dY=\frac{\partial}{\partial h}\int\Prob(Y|h,\Theta_m)dY=\frac{\partial}{\partial h}1=0\\
\int\frac{\partial^2}{\partial h^2}\Prob(Y|h,\Theta_m)dY=\frac{\partial}{\partial h}\int\frac{\partial}{\partial h}\Prob(Y|h,\Theta_m)dY=0
\end{align*}
and it follows that
\begin{align*}
\int\frac{\partial}{\partial h}\log\Prob(Y|h,\Theta_m)\ \Prob(Y|h,\Theta_m)dY=\int\frac{\partial}{\partial h}\Prob(Y|h,\Theta_m)dY=0
\end{align*}
and
\begin{equation}
\begin{aligned}
&\int\frac{\partial^2}{\partial h^2}\log\Prob(Y|h,\Theta_m)\ \Prob(Y|h,\Theta_m)dY\\
&=\frac{\partial}{\partial h}\int\frac{\partial}{\partial h}\log\Prob(Y|h,\Theta_m)\ \Prob(Y|h,\Theta_m)dY-\int\frac{\partial}{\partial h}\log\Prob(Y|h,\Theta_m)\frac{\partial}{\partial h}\Prob(Y|h,\Theta_m)dY\\
&=-\int\biggl(\frac{\partial}{\partial h}\log\Prob(Y|h,\Theta_m)\biggr)^2\Prob(Y|h,\Theta_m)dY
\end{aligned}
\label{identical equation}
\end{equation}
By the lemma and R4,
\begin{equation}
\begin{aligned}
|&\int\frac{\partial}{\partial h}\log\Prob(Y|h,\Theta_m)\ \Prob(Y|h,\tilde{\Theta}_m)dY|\\
&=|\int\frac{\partial}{\partial h}\log\Prob(Y|h,\Theta_m)\ (\Prob(Y|h,\tilde{\Theta}_m)-\Prob(Y|h,\Theta_m))dY|\\
&\leqslant2\big(\int|\frac{\partial}{\partial h}\log\Prob(Y|h,\Theta_m)|^2(\Prob(Y|h,\tilde{\Theta}_m)+\Prob(Y|h,\Theta_m))dY\big)^{\frac{1}{2}}d_{Hell}(\Prob(Y|h,\tilde{\Theta}_m), \Prob(Y|h,\Theta_m))\\
&\leqslant4\sqrt{2K}\varepsilon(m),
\end{aligned}
\end{equation}
and $\lim_{m\rightarrow\infty}4\sqrt{2K}\varepsilon(m) = 0$.

Thus $A_{n,m}$ is the mean of $n$ independent random vectors and the expected value of each random vectors is $0$ as $m$ goes to infinity. Therefore, by the Weak Law of Large Numbers(WLLN), $\lim_{n, m\rightarrow\infty}A_{n,m}=0$ in probability.

Let
\begin{align*}
B_{n,m}&=\frac{\partial^2}{\partial h^2}L(h,Y_1,\cdots,Y_n,\Theta_m)|_{h = \tilde{h}}\\
&=-\frac{1}{n}\sum_{i=1}^n\frac{\partial^2}{\partial h^2}\log\Prob(Y_i|h,\Theta_m^i)|_{h = \tilde{h}}\\
C_{n,m}&=H(Y_1,\cdots,Y_n, \tilde{h})\\
&=\frac{1}{n}\sum_{i=1}^n H(Y, \tilde{h})
\end{align*}

Combined with equation \ref{Taylor expansion}, we have
\begin{equation}
\begin{aligned}
0&=
\frac{\partial}{\partial h}L(h,Y_1,\cdots,Y_n,\Theta_m)|_{h = \hat{h}}\\
&=\frac{\partial}{\partial h}L(h,Y_1,\cdots,Y_n,\Theta_m)|_{h = \tilde{h}}+(\hat{h}-\tilde{h})\frac{\partial^2}{\partial h^2}L(h,Y_1,\cdots,Y_n,\Theta_m)|_{h = \tilde{h}}\\
&+\frac{1}{2}\zeta(\hat{h}-\tilde{h})^2H(Y_1,\cdots,Y_n, \tilde{h})\\
&=A_{n,m}+(\hat{h}-\tilde{h})\biggl(B_{n,m}+\frac{1}{2}\zeta(\hat{h}-\tilde{h})C_{n,m}\biggr)
\end{aligned}
\end{equation}

Thus
\begin{equation}
\lim_{m, n\rightarrow\infty}(\hat{h}-\tilde{h})\biggl(B_{n,m}+\frac{1}{2}\zeta(\hat{h}-\tilde{h})C_{n,m}\biggr) = 0\quad\text{in probability}.
\end{equation}
Analyzing as same as the $A_{n,m}$, we have
\begin{equation}
\begin{aligned}
&\int\frac{\partial^2}{\partial h^2}\log\Prob(Y|h,\Theta_m)\ \Prob(Y|h,\tilde{\Theta}_m)dY-\int\biggl(\frac{\partial}{\partial h}\log\Prob(Y|h,\Theta_m)\biggr)^2\Prob(Y|h,\Theta_m)dY\\
&=\int\frac{\partial^2}{\partial h^2}\log\Prob(Y|h,\Theta_m)\ \Prob(Y|h,\tilde{\Theta}_m)dY+\int\frac{\partial^2}{\partial h^2}\log\Prob(Y|h,\Theta_m)\ \Prob(Y|h,\Theta_m)dY\\
&\leqslant4\sqrt{2K}\varepsilon(m).
\end{aligned}
\end{equation}
According to $\lim_{m\rightarrow\infty}\varepsilon(m) = 0$ and assumption R3, we have
\begin{eqnarray*}
&\lim_{n, m\rightarrow\infty}B_{n,m} = K_{\infty}\quad\text{in probability},\\
&\lim_{n, m\rightarrow\infty}C_{n,m} = \E_Y[H(Y)]<\infty\quad\text{in probability}.
\end{eqnarray*}

According to the Slutsky Theorem,
\begin{align}
|\hat{h}-\tilde{h}| = \lim_{n,m\rightarrow\infty} \frac{|A_{n,m}|}{|B_{n,m}+\frac{1}{2}C_{n,m}\zeta(\hat{h}-\tilde{h})|}=0\quad\text{in distribution}.
\end{align}
Furthermore, based on central limit theorem,
\begin{align}
\begin{aligned}
\sqrt{N}|\hat{h}-\tilde{h}| &= \frac{\sqrt{N}|A_{n,m}|}{|B_{n,m}+\frac{1}{2}C_{n,m}\zeta(\hat{h}-\tilde{h})|}\\
\end{aligned}
\end{align}
converges to Gaussian distribution with variance of $1/{K_{\infty}}$.

\end{proof}

% % % % % % % % % % % % % % % % % % % % % % %
\paragraph*{S4 Appendix.}\label{S4_Appendix}
{\bf Prove of the Theorem~\ref{Thm:2}.}\\
This proof refers to the book~\cite{2015Data}.\\
\begin{proof}
Let $\rho$, $\rho'$ denote the Lebesgue densities on $\mu$, $\mu^{'}$, respectively. Then
\begin{equation*}
\begin{aligned}
\rho(x) &= \Prob(x|y_{1:\TT})\propto\Prob(y_{1:\TT}, x)\propto\Prob(y_{1:\TT}|x)\Prob(x)\\
\rho'(x) &= \Prob(x|y'_{1:\TT})\propto\Prob(y'_{1:\TT}, x)\propto\Prob(y'_{1:\TT}|x)\Prob(x)\\
\end{aligned}
\end{equation*}

As the equation $\ref{I_filter}$ shows,
\begin{equation*}
\begin{aligned}
\Prob(y_{1:\TT}|x)&=\prod_{t=1}^\TT\Prob(y_{t}|x)\propto\prod_{t=1}^\TT\exp(-I_{filter, t})\\
&\propto\prod_{t=1}^\TT\exp(-\frac{1}{2}|y_{t}-H(x)|_{\Gamma}^2-\frac{1}{2}|x-\hat{\mu}_{t}|_{\hat{C}_{t}}^2)\notag\\
&\propto\exp(-\sum_{t=1}^\TT\frac{1}{2}|y_{t}-H(x)|^2_{\Gamma})\\
\Prob(y'_{1:\TT}|x)&\propto\prod_{t=1}^\TT\exp(-I_{filter, t})\propto\exp(-\sum_{t=1}^\TT\frac{1}{2}|y'_{t}-H(x)|^2_{\Gamma})
\end{aligned}
\end{equation*}
Let $\rho_0(x) = \Prob(x)$ denote the Lebesgue densities on $\mu_0$,
\begin{align*}
\varrho(x) = \sum_{t=1}^\TT\frac{1}{2}|y_{t}-H(x)|^2_{\Gamma}, \\
\varrho'(x) = \sum_{t=1}^\TT\frac{1}{2}|y'_{t}-H(x)|^2_{\Gamma}.
\end{align*}
Then
\begin{equation*}
\begin{aligned}
\rho(x) = \frac{1}{Z}\exp(-\sum_{t=1}^\TT\frac{1}{2}|y_{t}-H(x)|^2_{\Gamma})\rho_0(x) = \frac{1}{Z}\exp(-\varrho(x))\rho_0(x),\\
\rho'(x) = \frac{1}{Z'}\exp(-\sum_{t=1}^\TT\frac{1}{2}|y'_{t}-H(x)|^2_{\Gamma})\rho_0(x) = \frac{1}{Z'}\exp(-\varrho'(x))\rho_0(x),
\end{aligned}
\end{equation*}
where
\begin{align*}
Z = \int\exp(-\sum_{t=1}^\TT\frac{1}{2}|y_{t}-H(x)|^2_{\Gamma})\rho_0(x)dx, \\
Z' = \int\exp(-\sum_{t=1}^\TT\frac{1}{2}|y'_{t}-H(x)|^2_{\Gamma})\rho_0(x)dx.
\end{align*}
Thus,
\begin{equation*}
\begin{aligned}
d_{Hell}(\mu, \mu')^2 =&~\frac{1}{2}\int|\sqrt{\rho(x)}-\sqrt{\rho'(x)}|^2dx,\\
=&~\frac{1}{2}\int |\frac{\sqrt{\exp(-\varrho(x))}}{\sqrt{Z}}- \frac{\sqrt{\exp(-\varrho'(x))}}{\sqrt{Z'}}|^2\rho_0(x)dx,\\
\leqslant&~\frac{1}{2}\int\frac{1}{Z}|\sqrt{\exp(-\varrho(x))}- \sqrt{\exp(-\varrho'(x))}|^2\rho_0(x)dx\\
&+\frac{1}{2}|\frac{1}{\sqrt{Z}}-\frac{1}{\sqrt{Z'}}|^2\int\exp(-\varrho'(x))\rho_0(x)dx,\\
=&~\frac{1}{2Z}\int|\exp(-\varrho(x)/2)- \exp(-\varrho'(x)/2)|^2\rho_0(x)dx\\
&+\frac{1}{2Z}|\sqrt{Z'}-\sqrt{Z}|^2.\\
\end{aligned}
\end{equation*}
And we assume that $y, y'$ are both contained in a ball of radius $r$ in the Euclidean norm on $\mathbb{R}^{\TT\times q}$. Thus, $\rho_0$ is bounded independently of $r$, $\rho$ and $\rho'$ are bounded dependently of $r$. Thus there exists $K_1$ depended on $r$ satisfy:

\begin{equation}
\begin{aligned}
\frac{1}{2Z}|\sqrt{Z'}-\sqrt{Z}|^2
&=\frac{1}{2Z}\frac{|Z-Z'|^2}{|\sqrt{Z}+\sqrt{Z'}|^2}\\
&\leqslant K_1|Z-Z'|^2\\
&=K_1|\int\exp(-\varrho(x))\rho_0(x)dx- \int\exp(-\varrho'(x))\rho_0(x)dx|^2\\
&\leqslant K_1\int|\exp(-\varrho(x))-\exp(-\varrho'(x))|^2\rho_0(x)dx
\end{aligned}
\label{z2rho}
\end{equation}

Since $\varrho(x)>0$ and $\varrho'(x)>0$, using the equation $\ref{z2rho}$ and the fact that $\exp(-x)$ and $\exp(-x/2)$ are both Lipschitz on $\R^+$, there exists $K_2$ depended on $r$ satisfy
\begin{equation}
\begin{aligned}
d_{Hell}(\mu, \mu')^2 \leq&~\frac{1}{2Z}\int|\exp(-\varrho(x)/2)- \exp(-\varrho'(x)/2)|^2\rho_0(x)dx\\
&+K_1\int|\exp(-\varrho(x))-\exp(-\varrho'(x))|^2\rho_0(x)dx\\
\leq&~K_2\int|\varrho(x)-\varrho'(x)|^2\rho_0(x)dx\\
\leq&~K_2\int|\frac{1}{2}\sum_{t=1}^\TT\left(|y_{t}-H(x_t)|^2_{\Gamma}-|y'_{t}-H(x_t)|^2_{\Gamma}\right)|^2\rho_0(x)dx\\
\leq&~\frac{K_2}{4}\int\biggl(\sum_{t=1}^\TT|y_{t}-y'_{t}|_{\Gamma}\times|y_{t}-H(x_t)+y_{t}-H(x_t)|_{\Gamma}\biggr)^2\rho_0(x)dx\\
\leq&~K_2\int\biggl(\sum_{t=1}^\TT|y_{t}-y'_{t}|^2_{\Gamma}\biggr)\times\biggl(\sum_{t=1}^\TT|\frac{y_{t}+y'_{t}}{2}-H(x_t)|^2_{\Gamma}\biggr)\rho_0(x)dx\\
\end{aligned}
\end{equation}
Finally, we note that since all norms are equivalent on finite-dimensional spaces, there is constant $K_3$ such that
\begin{equation}
\sum_{t=1}^\TT|y_{t}-y'_{t}|^2_{\Gamma}\leq K_3|y_{1:\TT}-y'_{1:\TT}|^2
\end{equation}
Thus,
\begin{equation}
\begin{aligned}
d_{Hell}(\mu, \mu')^2 \leq&~K_2\int K_3|y_{1:\TT}-y'_{1:\TT}|^2\times\sum_{t=1}^\TT|\frac{y_{t}+y'_{t}}{2}-H(x_t)|^2_{\Gamma}\rho_0(x)dx\\
\leq&~K_2\int K_3|y_{1:\TT}-y'_{1:\TT}|^2\times\sum_{t=1}^\TT\big(r^2+|H(x_t)|^2\big)\rho_0(x)dx\\
\leq&~K_2K_3r^2|y_{1:\TT}-y'_{1:\TT}|^2\int\sum_{t=1}^\TT\big(1+|H(x_t)|^2\big)\rho_0dx\\
=&~K_2K_3r^2|y_{1:\TT}-y'_{1:\TT}|^2\mathbb{E}_{\mu_0}w(x_{1:\TT})
\end{aligned}
\end{equation}
Since $\mathbb{E}_{\mu_0}w(x_{1:\TT}) < \infty$, there exists $c = c(r)$ such that for all $|y_{1:\TT}|, |y'_{1:\TT}| < r$,
\begin{align}
d_{Hell}(\mu, \mu')\leqslant c|y_{1:\TT}-y'_{1:\TT}|.
\end{align}
\end{proof}

% % % % % % % % % % % % % % % % % % % % % % %
% % % % % % % % % % % % % % % % % % % % % % %
\section*{Acknowledgments}
This work is jointly supported by the National Key R \& D Program of China (No. 2018AAA0100303), the National Natural Science Foundation of China (No. 62072111), the Shanghai Municipal Science and Technology Major Project (No. 2018SHZDZX01) and the ZHANGJIANG LAB.

%\nolinenumbers

\end{document}